\documentclass[11pt]{amsart}

\usepackage{epigamath-voisin}

\usepackage[english]{babel}

\numberwithin{equation}{section}

\usepackage[shortlabels]{enumitem}

\usepackage{amsmath,amssymb}

\usepackage[T1]{fontenc}
\usepackage[all]{xy}

\usepackage{babel}
\usepackage{amstext}
\usepackage{amsmath}
\usepackage{amsfonts}
\usepackage{latexsym}
\usepackage{ifthen}

\usepackage{xypic}
\xyoption{all}
\usepackage{multirow}
\usepackage{graphicx}
\usepackage{tikz-cd}
\usepackage{float}

\usepackage{url}

\usepackage{eurosym}

\newtheorem{theorem}{Theorem}[section]
\newtheorem{lemma}[theorem]{Lemma}
\newtheorem{proposition}[theorem]{Proposition}
\newtheorem{corollary}[theorem]{Corollary}

\theoremstyle{definition}

\newtheorem{definition}[theorem]{Definition}
\newtheorem{notation}[theorem]{Notation}
\newtheorem{setup}[theorem]{Set-up}

\theoremstyle{remark}

\newtheorem{remark}[theorem]{Remark}
\newtheorem{remarks}[theorem]{Remarks}
\newtheorem{question}[theorem]{Question}
\newtheorem*{remark*}{Remark}

\newcommand{\R}{\ensuremath{\mathbb{R}}}
\newcommand{\Q}{\ensuremath{\mathbb{Q}}}

\newcommand{\C}{\ensuremath{\mathbb{C}}}

\newcommand{\PP}{\ensuremath{\mathbb{P}}}

\newcommand{\holom}[3]{\ensuremath{#1\colon #2  \rightarrow #3}}
\newcommand{\fibre}[2]{\ensuremath{#1^{-1} (#2)}}

\def\shortbar{%
\smash{\scalebox{0.4}[1.0]{$-$}}}
\makeatother
      
\makeatletter
\newcommand{\xdashrightarrow}[2][]{\ext@arrow 0359\rightarrowfill@@{#1}{#2}}
\def\rightarrowfill@@{\arrowfill@@\relax\shortbar\dashrightarrow}
\def\arrowfill@@#1#2#3#4{%
  $\m@th\thickmuskip0mu\medmuskip\thickmuskip\thinmuskip\thickmuskip
   \relax#4#1
   \xleaders\hbox{$#4#2$}\hfill
   #3$%
}
\makeatother

\newcommand{\longdashrightarrow}{\xdashrightarrow{\hphantom{0pt}}}

\newcommand{\holomd}[3]{\ensuremath{#1\colon #2 \longrightarrow #3}}

\DeclareRobustCommand\longtwoheadrightarrow
    {\relbar\joinrel\twoheadrightarrow}

\makeatletter
\ifnum\@ptsize=0  \fi \ifnum\@ptsize=2  \fi 

\newcommand\sI{{\mathcal I}}

\newcommand\sO{{\mathcal O}}

\DeclareMathOperator*{\sing}{sing}

\DeclareMathOperator*{\red}{red}

\DeclareMathOperator*{\Pseff}{Pseff}

\DeclareMathOperator*{\node}{node}
\DeclareMathOperator*{\cusp}{cusp}

\newcommand{\ND}{\ensuremath{N_{\tilde D}}}
\newcommand{\Lnode}{\ensuremath{L_{\node}}}
\newcommand{\Lcusp}{\ensuremath{L_{\cusp}}}

\newcommand{\Cnode}{\ensuremath{C_{\node}}}
\newcommand{\CnodeS}{\ensuremath{C_{\node,S}}}
\newcommand{\Cnodei}{\ensuremath{C^i_{\node}}}
\newcommand{\Ccusp}{\ensuremath{C_{\cusp}}}

\newcommand{\NE}[1]{ \ensuremath{ \overline { \mbox{NE} }(#1)} }

\setlist[enumerate,1]{label={\rm(\alph*)}, ref={\rm\alph*}} 
\setlist[enumerate,2]{label={\rm(\roman*)}, ref={\rm\roman*}}

\YearArticle{2023} \EpigaArticleNr{3} \ReceivedOn{August 24, 2022}
\InFinalFormOn{March 9, 2023}
\AcceptedOn{March 27, 2023}

\title{The cotangent bundle of K3 surfaces of degree two}
\titlemark{The cotangent bundle of K3 surfaces of degree two}
\author{Fabrizio Anella}
\address{Sorbonne Universit\'e, IMJ-PRG, 4, place Jussieu, 75252 Paris Cedex 05, France}

\email{Anella@imj-prg.fr}
\author{Andreas H\"oring}
\address{Universit\'e C\^ote d'Azur, 
Parc Valrose
06108 Nice Cedex 02
France; CNRS, LJAD, France, Institut Universitaire de France}
\email{Andreas.Hoering@univ-cotedazur.fr}

\authormark{F.~Anella and A.~H\"oring}

\AbstractInEnglish{K3 surfaces have been studied from many points of
  view, but the positivity of the cotangent bundle is not well
  understood. In this paper we explore the surprisingly rich geometry
  of the projectivised cotangent bundle of a very general polarised K3
  surface $S$ of degree two. In particular, we describe the geometry of
  a surface $D_S \subset \PP(\Omega_S)$ that plays a similar role to
  the surface of bitangents for a quartic in $\PP^3$.}


\acknowledgement{The first-named author is sponsored by ERC Synergy Grant HyperK, Grant agreement ID 854361. He warmly thanks all the people involved in this project, for providing support and a stimulating environment. The second-named author thanks the Institut Universitaire de France and the A.N.R. project Foliage (ANR-16-CE40-0008) for providing excellent working conditions.}

\dedication{Dedicated to Claire Voisin, on the occasion of her birthday}

\begin{document}

%%%%%%%%%%%%%%%%%%%%%%%%%%%%%%%
% Title page
%%%%%%%%%%%%%%%%%%%%%%%%%%%%%%%

\maketitle

\begin{prelims}

\DisplayAbstractInEnglish

\bigskip

\DisplayKeyWords

\medskip

\DisplayMSCclass

\end{prelims}

%%%%%%%%%%%%%%%%%%%%%
% Table of Contents
%%%%%%%%%%%%%%%%%%%%%

\newpage

\setcounter{tocdepth}{1}

\tableofcontents

%%%%%%%%%%%%%%%%%%%%%
% Content begins here
%%%%%%%%%%%%%%%%%%%%%

\section{Introduction}

\subsection{Motivation}

The cotangent bundle of a K3 surface $S$ is well understood from the point of view of stability theory: we know that $\Omega_S$ is stable for every polarisation $L$. Moreover there are effective bounds guaranteeing that the restriction $\Omega_S \otimes \sO_C$ is stable for {\em every} irreducible curve $C \in |dL|$; see \cite{Hei06, Fey16}.  Nevertheless these stability results do not provide a complete description of the positivity properties of $\Omega_S$. In fact, the cotangent bundle of a K3 surface is never pseudoeffective (\textit{cf.} \cite{Nak04, BDPP13}, and \cite{HP19} for a more general result).

A way of measuring the negativity of the cotangent bundle is to describe the pseudoeffective cone of the projectivised cotangent bundle $\holom{\pi}{\PP(\Omega_S)}{S}$. Denote by $\zeta_S \rightarrow \PP(\Omega_S)$ the tautological class, and let $\alpha_S$ be a K\"ahler class on $S$. In our recent paper \cite{AH21} we showed that if $\alpha_S^2 \geq 8$, then $\zeta_S+\pi^* \alpha_S$ is pseudoeffective, the bound being optimal for a very general non-projective K3 surface; see \cite[Theorem~1.5]{AH21}.  The bound is also optimal for infinitely many, {\em but not all}, families of projective K3 surfaces with Picard number one, see \cite[Theorem~B]{GO20}, which leads us to a more geometric question.

\begin{question} \label{question1}
Let $S$ be a projective K3 surface, and let $L$ be an ample Cartier divisor on $S$.  Assume that $\zeta_S + \lambda \pi^* L$ is pseudoeffective for some $\lambda < \sqrt{\frac{8}{L^2}}$.  Can we relate the pseudoeffectivity of $\zeta_S + \lambda \pi^* L$ with the projective geometry of $(S, L)$?
\end{question}

In view of the arguments used by Lazi\'c and Peternell \cite{LP20}, we might ask even more boldly for a relation between $\Pseff(\PP(\Omega_S))$
and families of elliptic curves on $S$.  Gounelas and Ottem \cite{GO20} give a natural framework for these questions: the embedding
$$
\PP\left(\Omega_S\right) \subset S^{[2]}
$$
in the Hilbert square allows one to use results on the pseudoeffective cone of $S^{[2]}$ by Bayer and Macr\`{\i} \cite{BM14}. For example this approach allows one to recover the following classical result. 

\begin{theorem}[\textit{cf.} \protect{\cite{Tik80, Wel81}, \cite[Section~4]{GO20}}] \label{theorem-quartic}
Let $S \subset \PP^3$ be a smooth quartic surface with Picard number one,
and denote by $L$ the restriction of the hyperplane class. Then the surface of bitangents $U \subset \PP(\Omega_S)$ has class $6 \zeta_S +  8 \pi^* L$  and generates an extremal ray of\, $\Pseff(\PP(\Omega_S))$.
\end{theorem}

The goal of this paper is to give an analogue of Theorem~\ref{theorem-quartic} in the case of a very general polarised K3 surface $(S, L)$ of degree two; \textit{i.e.}\ $S$ is a K3 surface with Picard number one obtained as a two-to-one cover
$$
\holomd{f}{S}{\PP^2}
$$
with ramification divisor a smooth curve of genus ten.  Oguiso and Peternell \cite{OP96} observed that in this case, the projectivised cotangent bundle should have some exceptional properties that are not representative for general K3 surfaces.  The approach of Gounelas and Ottem allows one to determine the nef cone of $\PP(\Omega_S)$ \cite[Section~4.1]{GO20} but yields only that $\zeta_S + \pi^* 2L$ is pseudoeffective without determining the extremality in $\Pseff(\PP(\Omega_S)$.  Note that $\alpha_S:=2L$ has $\alpha_S^2=8$, so this situation corresponds exactly to the set-up of Question~\ref{question1}.

\subsection{Main results}

Let $S$ be a very general K3 surface of degree two, and let $\holom{f}{S}{\PP^2}$ be the double cover. Let $d \subset \PP^2$ be a line that is a simple tangent of the branch divisor $B$; then its preimage $C:=\fibre{f}{d}$ has a unique node, so the normalisation $\holom{n}{\tilde C}{C}$ is a smooth elliptic curve. The natural surjection
$$
n^* \Omega_S \longtwoheadrightarrow \Omega_{\tilde C}
$$
determines a morphism $\tilde C \rightarrow \PP(\Omega_S)$ which we call the canonical lifting of $C \in |L|$ (\textit{cf.} Section~\ref{subsection-canonical-liftings}).

\begin{theorem} \label{theorem-main1}
Let $(S, L)$ be a very general polarised K3 surface of degree two, and denote by $D_S \subset \PP(\Omega_S)$ the surface dominated by canonical liftings of singular elliptic curves in $|L|$. Then the normalisation of\, $D_S$ is a smooth $($non-minimal$)$ elliptic surface.  Moreover we have
$$
D_S \equiv 30 \zeta_S + 54 \pi^* L \equiv 30 \left(\zeta_S + 1.8 \pi^* L\right).
$$

\end{theorem}

The surface $D_S$ itself is very singular; the description of the normalisation occupies the larger part of this paper. We will see that $D_S$ contains a lot of geometric information about the K3 surface; for example it contains a curve isomorphic to the branch divisor $B$ in its non-normal locus.  While the surface $D_S$ plays a similar role to the surface of bitangents for the quartic surfaces, it does not generate an extremal ray in the pseudoeffective cone. 

\begin{theorem} \label{theorem-main2}
Let $(S, L)$ be a very general polarised K3 surface of degree two.
Then there exists a prime divisor $Z_S \subset \PP(\Omega_S)$
such that 
$$
Z_S \equiv a \left(\zeta_S + \lambda \pi^* \right)
$$
with $\lambda \leq 1.7952024$.\footnote{The exact statement is
$$
\lambda < \frac{15}{4} - 27 \frac{\left(1 - i \sqrt{3}\right)}{8 \left( \left(79 + 8 i \sqrt{5}\right)/3\right)^{(1/3)}} - \frac{1}{8} 3^{(2/3)} \left(1 + i \sqrt{3}\right) \left(79 + 8 i \sqrt{5}\right)^{(1/3)},
$$
but we refrain from working with such futile precision.}

Moreover let $Z_S \subset \PP(\Omega_S)$ be a prime divisor with this property.
Then the canonical liftings of\, $\CnodeS \subset \PP(\Omega_S)$ of the 324 rational curves $C$ in $|L|$ are contained in $Z_S$.
\end{theorem}

From a numerical point of view, both parts of this statement come as a surprise: 
\begin{itemize}
\item The non-nef locus of the divisor $\zeta_S+1.8 \pi^* L$ consists of a unique curve $R_S \subset \PP(\Omega_S)$. If we denote by $\holom{\mu_S}{Y}{\PP(\Omega_S)}$ the blow-up of this curve (\textit{cf.} Section~\ref{subsection-elementary-transform}), the strict transform $D \subset Y$ of $D_S$ generates an extremal ray in $\Pseff(Y)$.  Nevertheless $D_S$ itself is a big, non-nef divisor with $D_S^3<0$.
\item The divisor $Z_S$ satisfies $Z_S \cdot \CnodeS>0$, so we would not expect these curves to be contained in the stable base locus. This property will be a consequence of our investigation of the threefold $Y \rightarrow \PP(\Omega_S)$.
\end{itemize}

While the class of the surface $D_S$ does not generate an extremal ray in $\Pseff(\PP(\Omega_S))$, the detailed geometric study involved in the proof of Theorem~\ref{theorem-main1} allows one to give a pretty sharp estimate for the extremal ray.

\begin{theorem} \label{theorem-estimate-ray}
Let $(S, L)$ be a very general polarised K3 surface of degree two.  Assume that there exists a prime divisor $Z_S \subset \PP(\Omega_S)$ such that $Z_S \equiv a (\zeta_S + \lambda \pi^* L)$.
Then we have $\lambda \geq \frac{39}{22} = 1.7\overline{72}$.
\end{theorem}

The proofs of all our results are based on the analysis of the birational morphisms
\begin{equation} \label{diagram-short}
\xymatrix{
& Y \ar[rd]^{\mu_S} \ar[ld]_{\mu_P} & 
\\
\PP(f^* \Omega_{\PP^2}) & & \PP(\Omega_S) 
}
\end{equation}
which allow us to transfer information from the well-understood $\PP(f^* \Omega_{\PP^2})\subset S \times (\PP^2)^\vee$ to the much more mysterious $\PP(\Omega_S)$.  In particular this will allow us to describe the surface $D_S \subset \PP(\Omega_S)$ as the strict transform of the universal family of singular elements in $|L|$ (\textit{cf.} Lemma~\ref{lemma-D-DS}). The first step is Theorem~\ref{theorem-structure-barD}, where we show that $\bar D$, the normalisation of this universal family, is a smooth minimal elliptic surface. The second step is to show that the normalisation $\tilde D \rightarrow D_S$ is also a smooth surface (\textit{cf.} Theorem~\ref{theorem-structure-D}); as a consequence the birational morphism $\tilde D \rightarrow \bar D$ is simply a blow-up of 720 points.  Determining the rich geometry of the elliptic surface $\tilde D$ is the main technical contribution of this paper and a somewhat delicate task. With all this geometric information at hand, the proofs of
Theorems~\ref{theorem-main2} and~\ref{theorem-estimate-ray} follow without too much effort.

\subsection{Future directions: Return to the Hilbert square}
We will see in Section~\ref{subsection-elementary-transform} that $\PP(f^* \Omega_{\PP^2}) \rightarrow (\PP^2)^\vee$ can be identified with the universal family $\mathcal U \rightarrow |L|$ of divisors in the linear system $|L|$. This allows us to relate the morphisms in Diagram~\eqref{diagram-short} to a construction on the Hilbert scheme, \textit{cf.} \cite[Example~9]{Bak17}:
the Hilbert square $S^{[2]}$ contains a Lagrangian plane $\PP^2$ determined by mapping a point $z \in \PP^2$ to its preimage $[\fibre{f}{z}] \in S^{[2]}$.  Let
$$
S^{[2]} \longdashrightarrow X
$$
be the Mukai flop of this Lagrangian plane. Then $X$ is a hyperk\"ahler manifold that admits a Lagrangian fibration $X \rightarrow \PP^2$; in fact $X$ is the degree two compactified Jacobian of the universal family $\mathcal U \rightarrow |L|$. Moreover the graph of the Mukai flop is the relative Hilbert scheme $\mbox{Hilb}^2(\mathcal U/|L|)$.

The restriction of the blow-up $\mbox{Hilb}^2(\mathcal U/|L|) \rightarrow S^{[2]}$ to $\PP(\Omega_S)$ is our blow-up $Y \rightarrow \PP(\Omega_S)$ (\textit{cf.} \cite[Section~4.1]{GO20}), so we have a commutative diagram
$$
\xymatrix{
& & Y \ar @{^{(}->}[d] \ar[lldd]_{\mu_P} \ar[rrdd]^{\mu_S} & &
\\
& & \mbox{Hilb}^2(\mathcal U/|L|) \ar[ld] \ar[rd] & &
\\
\mathcal U = \PP(f^* \Omega_{\PP^2}) \ar[r] \ar[rd] & X \ar[d] & & S^{[2]} & \PP(\Omega_S)  \ar @{_{(}->}[l]
\\
& |L|\rlap{.} & & & 
}
$$
It is tempting to believe that a description of the pseudoeffective cone of the relative Hilbert scheme $\mbox{Hilb}^2(\mathcal U/|L|)$ allows one to shed further light on the geometry of $Y$ and ultimately determines the pseudoeffective cone of $\PP(\Omega_S)$.

\subsection*{Acknowledgments} We thank Thomas Dedieu for pointing out the relation between Theorem~\ref{theorem-structure-barD} and Teissier's theorem, \textit{cf.} Remark~\ref{remark-dedieu}. We thank the referee for the careful verification of our text. 

\section{Notation and general set-up}

We work over $\C$. For general definitions, we refer to \cite{Har77} and \cite{Voi02}.

All the schemes appearing in this paper are projective; manifolds and normal varieties will always be supposed to be irreducible.  For notions of positivity of divisors and vector bundles, we refer to Lazarsfeld's books \cite{Laz04a, Laz04b}.  Given two Cartier divisors $D_1, D_2$ on a projective variety $X$, we denote by $D_1 \simeq D_2$ (resp.\ $D_1 \equiv D_2$) the linear equivalence (resp.\ numerical equivalence) of the Cartier divisor classes, while $D_1=D_2$ is used for an equality of the cycles.  We will frequently identify an effective divisor with its cohomology class in $N^1(X)$, where $N^1(X):= \mbox{NS}(X) \otimes \R$ is the N\'{e}ron--Severi space of $\R$-divisors.

In the whole paper we will work in the following setting. 

\begin{setup} \label{setup}
Let $S$ be a very general polarised K3 surface of degree two, and let $L$ be the primitive polarisation on $S$. We denote by 
$$
\holomd{f}{S}{\PP^2}
$$ 
the double cover defined by the linear system $|L|$, so we have $L \simeq f^* H_1$, where $H_1$ is the hyperplane class on $\PP^2$. Since $|L| = f^* |H_1|$, every curve $C \in |L|$ is a double cover of a line $d \subset \PP^2$.

We denote by $B \subset \PP^2$ the branch locus of $f$ and by $R \subset S$ the ramification divisor.  By the ramification formula we know that $B$ is a sextic curve, so
$$
g(B) = 10, \qquad \deg \omega_B = 18.
$$
Since $f^* B = 2 R$, we see that the ramification divisor $R$ is an element of $|3L|$. 

Denote by 
$$
\holomd{\pi}{\PP(\Omega_S)}{S}
$$ 
the projectivisation of the cotangent bundle and by $\zeta_S$ the tautological class. We denote by $l \subset \PP(\Omega_S)$ a fibre of $\pi$.
\end{setup}

This set-up will become increasingly rich through a series of geometric constructions which we will summarise in Diagrams \eqref{diagram} and \eqref{bigdiagram}.

\begin{remark*} Many of our arguments are valid for an arbitrary smooth polarised K3 surface of degree two. However the assumption that $S$ is very general is necessary to ensure that the Picard number of $S$ is one. Moreover this implies that the branch curve $B \subset \PP^2$ satisfies the assumptions of Pl\"ucker's theorem, which is crucial in Section~\ref{subsection-surface-D} and in the proof of Theorem~\ref{theorem-structure-D}.
\end{remark*}

\section{Birational geometry of the projectivised cotangent bundle}

\subsection{Elementary transform of $\Omega_S$}
\label{subsection-elementary-transform}

Consider the projective plane $\PP^2$ and its hyperplane class $H_1$.  We denote by $\holom{p_1}{\PP(\Omega_{\PP^2})}{\PP^2}$ the projectivisation and by $\zeta_{\PP^2}$ its tautological class.  Let us recall some elementary facts: twisting the Euler sequence by $\sO_{\PP^2}(-1)$ and recalling that $T_{\PP^2} \simeq \Omega_{\PP^2} \otimes K_{\PP^2}^* \simeq \Omega_{\PP^2} \otimes \sO_{\PP^2}(3)$, we obtain an exact sequence
$$
0 \longrightarrow \sO_{\PP^2}(-1) \longrightarrow \sO_{\PP^2}^{\oplus 3} \longrightarrow \Omega_{\PP^2} \otimes \sO_{\PP^2}(2)
\longrightarrow 0.
$$
Thus the vector bundle $\Omega_{\PP^2} \otimes \sO_{\PP^2}(2)$ is globally generated, and $h^0(\PP^2, \Omega_{\PP^2} \otimes \sO_{\PP^2}(2))=3$.

The sequence above shows that there is an inclusion 
$$
\PP(\Omega_{\PP^2}) \subset \PP^2 \times (\PP^2)^\vee,
$$
and if we denote by $\holom{p_2}{\PP(\Omega_{\PP^2})}{(\PP^2)^\vee}$ the fibration defined by the projection on the second factor, it allows us to identify $\PP(\Omega_{\PP^2})$ with the universal family of lines on $\PP^2$; \textit{i.e.}\ for a point $y \in (\PP^2)^\vee$, the curve $p_1(\fibre{p_2}{y}) \subset \PP^2$ is the line corresponding to the point $y$.  The divisor class $\zeta_{\PP^2} + 2 p_1^* H_1$ defines a base-point-free linear system; in fact we have
$$
\zeta_{\PP^2}  + 2 p_1^* H_1 \simeq p_2^* H_2, 
$$
where $H_2$ is the hyperplane class on $(\PP^2)^\vee$.

From now on, we work in the setting of Set-up~\ref{setup}: 
let $\holom{f}{S}{\PP^2}$ be the double cover, and denote by
$$
\holomd{\tilde f}{\PP(f^* \Omega_{\PP^2})}{\PP(\Omega_{\PP^2})}
$$
the induced cover and by $\holom{p}{\PP(f^* \Omega_{\PP^2})}{S}$ the natural fibration.
We set 
$$
\zeta_P := \tilde f^* \zeta_{\PP^2}
$$
for the tautological class. The linear system
defined by $ \zeta_P + p^* 2L \simeq \tilde f^* (\zeta_{\PP^2}  + 2 p_1^* H_1)$
is base-point-free and 
defines a fibration 
$$
\holomd{q:= p_2 \circ \tilde f}{\PP(f^* \Omega_{\PP^2})}{(\PP^2)^\vee}.
$$
Since the pull-backs of lines $d \in |\sO(1)|$ correspond exactly to the elements of the linear system $|L|$, the fibration
$\holom{q}{\PP(f^* \Omega_{\PP^2})}{(\PP^2)^\vee}$ is the universal family of curves in the linear system $|L|$. 

The key to our investigation is the birational map
$$
\PP(f^* \Omega_{\PP^2}) \longdashrightarrow \PP(\Omega_S)
$$
which can be explicitly described as follows: since $R$ is the ramification divisor of the double cover $f$, we have
\begin{equation} \label{OmegaSsplit}
\Omega_S \otimes \sO_R \simeq \omega_R \oplus \sO_R(-R), 
\end{equation}
and the relative cotangent sheaf $\Omega_f$ is isomorphic to $\sO_R(-R)$. Now consider the exact sequence
\begin{equation} \label{tangent-sequence}
0 \longrightarrow f^* \Omega_{\PP^2} \longrightarrow \Omega_S \longrightarrow \Omega_f \longrightarrow 0.
\end{equation}
The vector bundles $f^* \Omega_{\PP^2}$ and $\Omega_S$ are isomorphic in the complement of $R$, and we will see that the indeterminacy locus of $\PP(f^* \Omega_{\PP^2}) \dashrightarrow \PP(\Omega_S)$ (resp.\ its inverse) is also a curve $R \simeq R_P \subset \PP(f^* \Omega_{\PP^2})$ (resp.\ $R \simeq R_S \subset \PP(\Omega_S))$. Geometrically the curve $R_P$ corresponds to the cotangent bundle of the ramification divisor, while $R_S$ corresponds to the conormal bundle.

More formally, we make the following construction: since $\Omega_f$ is a rank one bundle with support on the Cartier divisor $R$, we can see $f^* \Omega_{\PP^2}$ as the strict transform of $\Omega_S$ along $\Omega_f$ (\textit{cf.} \cite[Theorem~1.3]{Mar72} for the terminology).  Let
$$
R_S := \PP(\Omega_f) \subset \PP(\Omega_S)
$$
be the curve defined by the surjection $\Omega_S \rightarrow \Omega_f$. Then denote by $\holom{\mu_S}{Y}{\PP(\Omega_S)}$ the blow-up along this curve and by $E_S$ its exceptional divisor.  Denote by $E_P$ the strict transform of the divisor $\fibre{\pi}{R} \subset \PP(\Omega_S)$. Note that since $\PP(\Omega_f) \subset \fibre{\pi}{R}$ is a Cartier divisor, we have
\begin{equation} \label{isomEP}
E_P \simeq \fibre{\pi}{R} \simeq \PP(\Omega_S \otimes \sO_R) \simeq \PP(\omega_R \oplus \sO_R(-R)), 
\end{equation}
where in the last step, we used \eqref{OmegaSsplit}.

By the relative contraction theorem applied to the morphism $\pi \circ \mu_S$, there exists a birational morphism contracting the divisor $E_P$ onto a curve.
By \cite[Theorem~1.3]{Mar72} this contraction can be described more explicitly:
the image of $f^* \Omega_{\PP^2} \otimes \sO_R \rightarrow \Omega_S \otimes \sO_R$ is isomorphic to $\omega_R$, so we have a surjection $f^* \Omega_{\PP^2} \rightarrow \omega_R$. Thus 
$$
R_P := \PP(\omega_R) \subset \PP(f^* \Omega_{\PP^2})
$$
is a curve, and the blow-up along $R_P$ gives a morphism $\holom{\mu_P}{Y}{\PP(f^* \Omega_{\PP^2})}$ such that the following diagram commutes:
\begin{equation} \label{diagram}
\xymatrix{
& E_P \ar[ld]  \ar @{^{(}->}[r] & Y \ar[rd]^{\mu_S} \ar[ld]_{\mu_P} &  E_S \ar[rd]  \ar @{_{(}->}[l] &
\\
R_P   \ar @{^{(}->}[r]  & \PP(f^* \Omega_{\PP^2}) \ar[ld]_q \ar[d]^p & & \PP(\Omega_S) \ar[d]^\pi & R_S \ar @{_{(}->}[l]
\\
(\PP^2)^\vee = |L|  & S \ar[rr]^= \ar[d]^f & & S &
\\
& \PP^2\rlap{.} & & &
}
\end{equation}
Arguing as above we see that
\begin{equation} \label{isomES}
E_S \simeq \fibre{p}{R} \simeq \PP(f^* \Omega_{\PP^2} \otimes \sO_R) \simeq \PP(\Omega_{\PP^2} \otimes \sO_B),
\end{equation}
where in the last step, we used that $f|_R$ is an isomorphism onto $B$.
We deduce from \cite[Theorem~1.1]{Mar72} that
\begin{equation} \label{transformtautological}
\mu_P^* \zeta_P = \mu_S^* \zeta_S - E_S, \qquad \mu_S^* \zeta_S = \mu_P^* \zeta_P - E_P + 3 (p \circ \mu_P)^* L.
\end{equation}

\begin{lemma} \label{lemma-intersectionsY-1}
In the situation summarised by Diagram~\ref{diagram}, we have the following intersection numbers:
$$
E_S^3=18, \qquad E_P^3=-72.
$$
Set $\tilde R := E_P \cap E_S$. Then we have
$$
E_S \cdot \tilde R = -36, \qquad E_P \cdot \tilde R = 54
$$
and
$$
(\tilde R)^2_{E_S}=54, \qquad (\tilde R)^2_{E_P}=-36.
$$
\end{lemma}

\begin{remark*}
We denote by $(\tilde R)^2_{E_S}$ (resp.\ $(\tilde R)^2_{E_P}$) the self-intersection of the curve $\tilde R$, seen as a divisor in $E_S$ (resp.~$E_P$).
\end{remark*}

\begin{proof}
We describe $E_S$ and $E_P$ in terms of the blow-up: by construction one has $E_S \simeq \PP(N^*_{R_S/\PP(\Omega_S)})$ and $\sO_{E_S}(E_S) \simeq \sO_{\PP(N^*_{R_S/\PP(\Omega_S)})}(-1)$. We have an exact sequence
$$
0 \longrightarrow N^*_{\fibre{\pi}{R}/\PP(\Omega_S)} \otimes \sO_{R_S} \longrightarrow N^*_{R_S/\PP(\Omega_S)} \longrightarrow N^*_{R_S/\fibre{\pi}{R}} \longrightarrow 0.
$$
Since $R \in |3L|$, we have $N^*_{\fibre{\pi}{R}/\PP(\Omega_S)} \otimes \sO_{R_S} \simeq \sO_R(-3L)$, where we have identified $R_S \simeq R$. Since $R_S$ corresponds to the quotient
$\Omega_S \otimes \sO_R \rightarrow \sO_R(-3L)$, an adjunction computation shows $N^*_{R_S/\fibre{\pi}{R}} \simeq \sO_R(6L)$. Thus we have an extension
$$
0 \longrightarrow \sO_R(-3L) \longrightarrow N^*_{R_S/\PP(\Omega_S)} \longrightarrow \sO_R(6L) \longrightarrow 0;
$$
in particular $\det N^*_{R_S/\PP(\Omega_S)} \simeq \sO_R(3L)$ and
$$
E_S^3 = c_1(\sO_{\PP(N^*_{R_S/\PP(\Omega_S)})}(-1))^2 = \deg N^*_{R_S/\PP(\Omega_S)} = \deg \sO_R(3L) = 18.
$$
Now recall that by \eqref{isomES}, we also have $E_S \simeq \PP(f^* \Omega_{\PP^2} \otimes \sO_R)$, so there exists a line bundle $M \rightarrow R$ such that $N^*_{R_S/\PP(\Omega_S)} \simeq f^* \Omega_{\PP^2} \otimes \sO_R \otimes M$. Since $\det f^* \Omega_{\PP^2} \otimes \sO_R \simeq \sO_R(-3L)$, we obtain
$$
N^*_{R_S/\PP(\Omega_S)} \simeq f^* \Omega_{\PP^2} \otimes \sO_R(3L).
$$
Now recall  that $\mu_P$ is the blow-up along the curve $R_P$ corresponding to the quotient $f^* \Omega_{\PP^2} \otimes \sO_R \rightarrow \omega_R \simeq \sO_R(3L)$. Then the curve $\tilde R$ can be identified with $R_P$ under the isomorphism $\mu_P|_{E_S} \colon E_S \rightarrow \PP(f^* \Omega_{\PP^2} \otimes \sO_R)$, and we have
\begin{align*}
E_S \cdot \tilde R &= c_1(\sO_{\PP(N^*_{R_S/\PP(\Omega_S)})}(-1)) \cdot \tilde R
=
c_1(\sO_{\PP(f^* \Omega_{\PP^2} \otimes \sO_R \otimes \sO_R(3L))}(-1)) \cdot \tilde R
\\
&=
- \left(
c_1(\sO_{\PP(f^* \Omega_{\PP^2} \otimes \sO_R}(1)) \cdot R_P
+
\deg \sO_R(3L)
\right)
= - \deg \sO_R(6L) = - 36.
\end{align*}
We repeat the argument for $E_P$: by construction one has $E_P \simeq \PP(N^*_{R_P/\PP(f^* \Omega_{\PP^2})})$ and $\sO_{E_P}(E_P) \simeq \sO_{N^*_{R_P/\PP(f^* \Omega_{\PP^2})}}(-1)$. We have an exact sequence
$$
0 \longrightarrow N^*_{\fibre{p}{R}/\PP(f^* \Omega_{\PP^2})} \otimes \sO_{R_P} \longrightarrow N^*_{R_P/\PP(f^* \Omega_{\PP^2})} \longrightarrow N^*_{R_P/\fibre{p}{R}} \longrightarrow 0.
$$
Since $R \in |3L|$, we have $N^*_{\fibre{p}{R}/\PP(f^* \Omega_{\PP^2})} \otimes \sO_{R_P} \simeq \sO_R(-3L)$, where we identified $R_P$ with $R$. Since $R_P$ corresponds to the quotient $f^* \Omega_{\PP^2} \otimes \sO_R \rightarrow \omega_R \simeq \sO_R(3L)$, an adjunction computation shows $N^*_{R_P/\fibre{p}{R}} \simeq \sO_R(-9L)$. Thus we have an extension
$$
0 \longrightarrow \sO_R(-3L) \longrightarrow N^*_{R_P/\PP(f^* \Omega_{\PP^2})} \longrightarrow \sO_R(-9L) \longrightarrow 0;
$$
in particular $\det N^*_{R_P/\PP(f^* \Omega_{\PP^2})} \simeq \sO_R(-12L)$ and
$$
E_P^3 = c_1(\sO_{\PP(N^*_{R_P/\PP(f^* \Omega_{\PP^2})})}(-1))^2 = \deg N^*_{R_P/\PP(f^* \Omega_{\PP^2})} = \deg \sO_R(-12L) = -72.
$$
Now recall that by \eqref{isomEP}, we also have $E_P \simeq \PP(\Omega_S \otimes \sO_R)$, so there exists a line bundle $M \rightarrow R$ such that $N^*_{R_P/\PP(f^* \Omega_{\PP^2})} \simeq \Omega_S \otimes \sO_R \otimes M$. Since $\det \Omega_S \otimes \sO_R \simeq \sO_R$, we finally obtain
$$
N^*_{R_P/\PP(f^* \Omega_{\PP^2})} \simeq \Omega_S \otimes \sO_R \otimes \sO_R(-6L).
$$
Now recall that $\mu_S$ is the blow-up along the curve $R_S$ corresponding to the quotient $\Omega_S \otimes \sO_R \rightarrow \sO_R(-3L)$. Then the curve $\tilde R$ is identified with $R_S$ under the isomorphism $\mu_S|_{E_P} \colon E_P \rightarrow \PP(\Omega_S \otimes \sO_R)$, and we have
\begin{align*}
E_P \cdot \tilde R & = c_1(\sO_{\PP(N^*_{R_P/\PP(f^* \Omega_{\PP^2})})}(-1)) \cdot \tilde R
=
c_1(\sO_{\PP(\Omega_S \otimes \sO_R \otimes \sO_R(-6L))}(-1)) \cdot \tilde R
\\&
=
- \left(
c_1(\sO_{\PP(\Omega_S \otimes \sO_R}(1)) \cdot R_S
+
\deg \sO_R(-6L)
\right)
= - \deg \sO_R(-9L) = 54.
\end{align*}
In order to determine the self-intersection of $\tilde R$, recall
that the self-intersection of a curve is invariant under isomorphism.  The isomorphism 
$\mu_P|_{E_S} \colon E_S \rightarrow \PP(f^* \Omega_{\PP^2} \otimes \sO_R)$ sends $\tilde R$ to $R_P$; thus we obtain
$$
(\tilde R)^2_{E_S} = (R_P)^2_{\fibre{p}{R}} = \deg N_{R_P/\fibre{p}{R}} = \deg \sO_R(9L) = 54.
$$
Since the isomorphism $\mu_S|_{E_P} \colon E_P \rightarrow \PP(\Omega_S \otimes \sO_R)$ sends $\tilde R$ to $R_S$, one has
$$
(\tilde R)^2_{E_P} = (R_S)^2_{\fibre{\pi}{R}} = \deg N_{R_S/\fibre{\pi}{R}} = \deg \sO_R(-6L) = -36.
$$
\end{proof}

\begin{corollary} \label{corollarypseffES}
In the situation summarised by Diagram~\ref{diagram}, denote by $l_S \subset E_S$ a fibre of the projection
$$
E_S \simeq \PP(f^* \Omega_{\PP^2} \otimes \sO_R) \longrightarrow R_S.
$$
Then the pseudoeffective cone of\, $E_S$ coincides with the nef cone, and its extremal rays are generated by the classes $l_S$ and $\tilde R-27 l_S$.
In particular if\, $\holom{\tau}{E_S}{W}$ is a generically finite morphism onto a surface $W$, it is finite.
\end{corollary}

\begin{proof}
By \eqref{isomES} we have $E_S \simeq \PP(\Omega_{\PP^2} \otimes \sO_B)$.  Since the K3 surface $S$ is general, the sextic curve $B$ is general in its linear system, so by \cite[Theorem~1.2]{Fle84} the restricted vector bundle $\Omega_{\PP^2} \otimes \sO_B$ is semistable.  Thus by \cite[Section~1.5.A]{Laz04a} the pseudoeffective cone of $E_S$ coincides with the nef cone. Since a nef divisor with positive self-intersection is big, this also shows that the nef cone coincides with the positive cone of $E_S$.

The curve $\tilde R \subset E_S$ is a section, and by Lemma~\ref{lemma-intersectionsY-1} we have $(\tilde R)^2_{E_S}=54$. Thus $\tilde R-27 l_S$ is a generator of the second extremal ray of the positive cone.

The last statement follows by observing that a curve contracted by a generically finite morphism of surfaces has negative self-intersection.
\end{proof}

\begin{remark} \label{remark-zeta-twoL}
We saw at the beginning of this subsection that the vector bundle $f^* \Omega_{\PP^2} \otimes \sO_S(2L)$ is globally generated and $h^0(S, f^* \Omega_{\PP^2} \otimes \sO_S(2L)) \geq 3$.  Using the projection formula and $f_* \sO_S \simeq \sO_{\PP^2} \oplus \sO_{\PP^2}(-3)$, we see that equality holds.

Twisting the tangent sequence \eqref{tangent-sequence} by $\sO_S(2L)$ and using that $\omega_f \simeq \sO_R(-R) \simeq \sO_R(-3L)$, we obtain that
$$
h^0(\PP(\Omega_S), \zeta_S + \pi^* 2L) = h^0(S, \Omega_S \otimes \sO_S(2L))=3, 
$$
and the linear system $|\zeta_S + \pi^* 2L|$ is globally generated in the complement of $R_S = \PP(\omega_f)$. In fact since the tangent map has rank one along $R$, we see that the scheme-theoretic base locus of $|\zeta_S + \pi^* 2L|$ is the reduced curve $R_S$.  Thus the blow-up $R_S$ resolves the base locus, and
$$
\mu_S^* (\zeta_S + \pi^* 2L) = \mu_P^*(\zeta_P + p^* 2L) + E_S = (q \circ \mu_P)^* H_2 +E_S,
$$
where $H_2$ is the hyperplane class on $(\PP^2)^\vee$.
\end{remark}

\begin{lemma} \label{lemma-intersectionsY-2}
In the situation summarised by Diagram~\ref{diagram}, we set
$$
\holomd{\psi:=q \circ \mu_P}{Y}{(\PP^2)^\vee}
\qquad\text{and}\qquad
M:= \psi^* H_2 = \mu_S^* (\zeta_S + \pi^* 2L) - E_S
$$ 
where $H_2$ is the hyperplane class on $(\PP^2)^\vee$. Then we have the following intersection numbers:
\begin{gather*}
  E_S^3=18, \ \ \ E_S^2 \cdot E_P=-36, \ \ \ E_S \cdot E_P^2=54, \ \ \ E_P^3=-72,\\
M^3=0,\ \ \ M^2 \cdot E_S=6, \ \ \ M  \cdot E_S^2=-12, \\
M^2 \cdot E_P=0,  \ \ \ M \cdot E_P^2=-30, \ \ \ M \cdot E_S \cdot E_P=30,\\
(E_P+E_S) \cdot E_S^2=-18, \ \ \ (E_P+E_S) \cdot E_P \cdot E_S=18, \ \ \ (E_P+E_S)\cdot E_P^2=-18,\\
M^2 \cdot (p \circ \mu_P)^* L=2, \ \ \ M \cdot (p \circ \mu_P)^* L^2 = 2,\\
(p \circ \mu_P)^* L \cdot E_P^2 = -6, \ \ \ (p \circ \mu_P)^* L^2 \cdot E_P=0.
\end{gather*}
\end{lemma}

\begin{proof} The result follows from Lemma~\ref{lemma-intersectionsY-1} 
and standard computations involving the geometric construction. In particular observe that
$$
E_P+E_S= (\pi \circ \mu_S)^* R =  (p \circ \mu_P)^* R,
$$
so $E_P+E_S \equiv 3 (p \circ \mu_P)^* L$. 

As an example, let us show that $M^2 \cdot E_S=6$: note that $\mu_P$ maps $E_S$ isomorphically onto $\PP(f^* \Omega_{\PP^2} \otimes \sO_R)$.  A fibre of $p$ maps onto a line in $(\PP^2)^\vee$, so $M$ has degree one on the fibres of $E_S \rightarrow S$. The curve $\tilde R \simeq R_P$ maps birationally onto $R^\vee \subset (\PP^2)^\vee$. Since $R^\vee$ has degree 30 by Pl\"ucker's formula (see Section~\ref{subsection-surface-D} below), we obtain $M \cdot \tilde R=30$.  Since $(\tilde R)^2_{E_S}=54$ by Lemma~\ref{lemma-intersectionsY-1}, this shows that $M|_{E_S} \equiv \tilde R - 24 l_S$, and thus $M^2 \cdot E_S = (M|_{E_S})^2=6$.
\end{proof}

\begin{remark*} 
%\label{restrict-ES}
For the convenience of the reader, let us summarise that by the argument above and Lemma~\ref{lemma-intersectionsY-1}, we have
$$
M|_{E_S} \equiv \tilde R - 24 l_S, \qquad  E_P|_{E_S} = \tilde R, \qquad E_S|_{E_S} \equiv - \tilde R + 18 l_S.
$$
\end{remark*}

\subsection{The surface $\boldsymbol{D}$}
\label{subsection-surface-D}

Let $B^\vee \subset (\PP^2)^\vee$ be the dual curve of the sextic curve $B \subset \PP^2$. Since $B$ is a general sextic, we can apply Pl\"ucker's formulas to see that
$$
\deg B^\vee = 30
$$
and $B^\vee$ has exactly 324 nodes (resp.\ 72 simple cusps) corresponding to bitangent lines (resp.\ inflection lines) and no other singularities.  By definition the dual curve parametrises lines $d \subset \PP^2$ that are tangent to $B$ in at least one point. Since the preimage $\fibre{f}{d}$ of a line is singular if and only if it is tangent to $B$ in at least one point, we see that $B^\vee$ naturally parametrises the singular elements of $|L|$.  We denote by
$$
R^\vee \subset |L| 
$$
the image of $B^\vee$ under the isomorphism $(\PP^2)^\vee \simeq |L|$.  Since the elements of $|L|$ have arithmetic genus two. it is not difficult to see that
\begin{itemize}
\item a smooth point $t \in R^\vee$ parametrises a curve $C_t$ with exactly one node, so the normalisation is an elliptic curve;
\item a node $t  \in R^\vee$   parametrises a curve $C_t$ with exactly two nodes, so the normalisation is a rational curve;
\item a cusp $t  \in R^\vee$   parametrises a curve $C_t$ with one simple cusp, so the normalisation is an elliptic curve.
\end{itemize}
The surface $\fibre{q}{R^\vee} \subset \PP(f^* \Omega_{\PP^2})$ identifies with the universal family of singular elements in $|L|$; in particular, this surface is not normal. 

\begin{lemma} \label{lemma-singularities-D}
We have $R_P \subset (\fibre{q}{R^\vee})_{\sing}$, and $R_P$ is the unique irreducible component of\, $(\fibre{q}{R^\vee})_{\sing}$ mapping onto $R^\vee$. Moreover $\fibre{q}{R^\vee}$ has a nodal singularity in a point of $R_P$ that maps onto a smooth point on $R^\vee$.
\end{lemma}

\begin{proof}
For a smooth point $t \in R^\vee$, the curve $C_t \subset |L|$ has a unique singular point $x$, so it is clear that $(\fibre{q}{R^\vee})_{\sing}$ has exactly one irreducible component mapping onto $R^\vee$. We claim that the point $x \in C_t \subset \PP(f^* \Omega_{\PP^2})$ is on the curve $R_P$.

{\em Proof of the claim.} Let $d \subset \PP^2$ be a line, and let $\tilde d \subset \PP(\Omega_{\PP^2})$ be the lifting defined by the canonical quotient $\Omega_{\PP^2} \otimes \sO_d \rightarrow \omega_d$. Then $\tilde d$ is the fibre of the fibration $\holom{p_2}{\PP(\Omega_{\PP^2})}{(\PP^2)^\vee}$ over the point $[d] \in (\PP^2)^\vee$; hence $C_t=\fibre{\tilde f}{\tilde d}$.

Let $B_P \subset \PP(\Omega_{\PP^2})$ be the curve defined by the canonical quotient
$\Omega_{\PP^2} \otimes \sO_B \rightarrow \omega_B$. Since the curve $R_P$
corresponds to the quotient
$$
\Omega_{\PP^2} \otimes \sO_B \simeq f^* \Omega_{\PP^2} \otimes \sO_R \longrightarrow \omega_R,
$$
we have a set-theoretical equality $R_P = \fibre{\tilde f}{B_P}$. 

Now recall that the point $x \in C_t = \fibre{f}{d}$ is singular if and only if
$$
T_{d,f(x)}
=  T_{B, f(x)} \subset T_{\PP^2, x}.
$$
Thus the curves $\tilde d$ and $B_P$ intersect over $f(x)$. This proves the claim.

In order to see that $\fibre{q}{R^\vee}$ has a nodal singularity over the smooth points of $R^\vee$, 
we just observe that $\fibre{q}{R^\vee}=\tilde f^* p_2^* B^\vee$.
Since $p_2^* B^\vee$ is a $\PP^1$-bundle over $B^\vee$, the statement follows by considering the intersection of a general $\PP^1$ with the branch divisor of the double cover $\tilde f$ (see the proof of Theorem~\ref{theorem-structure-D} for a more refined description of the singularities).
\end{proof}

We now introduce the main object of our study.

\begin{proposition} \label{proposition-class-D}
Let $D \subset Y$ be the strict transform of the surface $\fibre{q}{R^\vee} \subset \PP(f^* \Omega_{\PP^2})$.  Using the notation of Lemma~\ref{lemma-intersectionsY-2}, we have
\begin{equation} \label{class-D}
D=30M-2E_P.
\end{equation}
\end{proposition}

\begin{proof}
Since $R^\vee \subset |L|$ has degree $30$, we have $(q \circ \mu_P)^* R^\vee = 30 M$.  By Lemma~\ref{lemma-singularities-D} the surface $\fibre{q}{R^\vee}$ has multiplicity two along the curve $R_P$; thus its strict transform has class $30M-2E_P$.
\end{proof}

Since $D=30M-2E_P$ the following statement follows from Lemma~\ref{lemma-intersectionsY-2} by elementary computations.

\begin{lemma} \label{lemma-intersectionsY-3}
In the situation summarised by Diagram~\ref{diagram}, we have the following intersection numbers:
\begin{gather*}
  D^3=-10224, \ \ \ D^2 \cdot E_S=2016,  \ \ \  D \cdot E_S^2=-288,\\
D^2 \cdot E_P=3312, \ \ \ D\cdot E_P^2=-756, \ \ \ D\cdot E_P \cdot E_S=792,\\
(E_P+E_S) \cdot D^2=5328,\ \ \ (E_P+E_S) \cdot D \cdot E_P=36,\ \ \ (E_P+E_S)\cdot D \cdot E_S=504.
\end{gather*}
\end{lemma}

The prime divisors $E_P, E_S, D$ form a basis of $N^1(Y)$ that we will use for our computations in the later sections. The class $M=\psi^* c_1(\sO_{(\PP^2)^\vee}(1))$ will also be useful due to its simple geometric interpretation. For the convenience of the reader, we summarise their intersections with a dual base. 

\begin{lemma} \label{lemma-intersectionsY-4}
In the situation summarised by Diagram~\ref{diagram}, denote by $l_P$ $($resp.\ $l_S)$ an exceptional curve of the blow-up $\mu_P$ $($resp.\ $\mu_S)$, and by $l_D$ a general fibre of the fibration
$$
(q \circ \mu_P)_D\colon D \longrightarrow R^\vee.
$$
Then 
we have the following intersection numbers:
\begin{gather*}
E_S \cdot l_P=1, \ \ \  E_P \cdot l_P = -1, \ \ \  M \cdot l_P = 0, \ \ \  D \cdot l_P = 2,\\
E_S \cdot l_S=-1, \ \ \  E_P \cdot l_S = 1, \ \ \  M \cdot l_S = 1, \ \ \  D \cdot l_S = 28,\\
E_S \cdot l_D=4, \ \ \  E_P \cdot l_D = 2, \ \ \  M \cdot l_D = 0, \ \ \  D \cdot l_D = -4.
\end{gather*}
\end{lemma}

\begin{proof}
First note  that the intersection numbers involving $D$ follow from the other numbers and \eqref{class-D}.  Also note  that the intersection numbers for the curves $l_S$ and $l_P$ with $E_S, E_P, M$ are straightforward from the construction of the elementary transform, summarised in Diagram \eqref{diagram}.

Thus we are left to compute the intersection numbers of $l_D$. Recall that, by definition, $D \subset Y$ is the strict transform of $\fibre{q}{R^\vee} \subset \PP(f^* \Omega_{\PP^2})$.  Thus $l_D$ is the strict transform of a general fibre $C_t$ of $\fibre{q}{R^\vee} \subset R^\vee$, so it is clear that $M \cdot l_D=0$.  The fibre $C_t$ has a nodal singularity in its intersection point with the curve $R_P$; the natural map $\PP(f^* \Omega_{\PP^2}) \rightarrow S$ maps $R_P$ (resp.\ $C_t$) isomorphically onto $R$ (resp.\ the corresponding curve in $|L|$).  Thus each branch of $C_t$ meets $R_P$ transversally since this holds for the branches' images in $S$.  This shows that $E_P \cdot l_D=2$. Since $E_P + E_S = (p \circ \mu_P)^* 3 L$ and $p(l_D) \in |L|$, we deduce that $E_S \cdot l_D=4$.
\end{proof}

\subsection{Canonical liftings}
\label{subsection-canonical-liftings}

In this short subsection, we relate the surface $D$ to the surface $D_S$ appearing in the introduction.

\begin{definition} \label{definitioncanonicallifting}
Let $S$ be a smooth projective surface, and let $C \subset S$ be an irreducible curve. Denote by $\holom{n}{\tilde C}{C}$ the normalisation and by
$$n^* \Omega_S \longrightarrow Q_C$$
the image of the cotangent map
$n^* \Omega_S \rightarrow \omega_{\tilde C}$.

The canonical lifting of $C$ to $\PP(\Omega_S)$ is the image of the morphism
$\holom{\tilde n}{\tilde C}{\PP(\Omega_S)}$ corresponding to the line bundle $Q_C$. We denote this curve by $C_S \subset \PP(\Omega_S)$.
\end{definition}

\begin{remarks} \label{remarkscanonicallifting}\leavevmode
\begin{enumerate}
\item If the curve $C$ is singular, the birational map $C_S \rightarrow C$ is not necessarily  an isomorphism.
\item We have $Q_C \simeq \omega_{\tilde C}$ if and only if $C$ is immersed.
\item If $C \subset S$ is a nodal curve, then $C_S \subset \PP(\Omega_S)$ is a smooth curve: since $\pi \circ \tilde n = n$, the morphism is immersive, so we only have to show that it is one-to-one. Recalling that $\PP(\Omega_S) = \mathbf P(T_S)$ (the projective bundle of lines in $T_S$), we see from the definition of the canonical lifting that $\tilde n$ maps a point $t \in \tilde C$ to the point $[T n(T_{C,t})] \in P(T_{S, n(t)})$, where $T n$ is the tangent map. Yet if $t_1, t_2$ are two points in $\tilde C$ such that $n(t_1)=n(t_2)$, then $[T n(T_{C,t_1})] \neq [T n (T_{C,t_2})]$ since the curve is nodal. Thus $\tilde n(t_1) \neq \tilde n(t_2)$.  This also shows that $C_S$ meets the fibre $\fibre{\pi}{n(t_i)} \simeq \mathbf P(T_{S,n(t_i)})$ transversally in two points.
\item If $C \subset S$ is a curve such that the unique singular point is a simple cusp $p$, the canonical lifting $C_S \subset \PP(\Omega_S)$ is a smooth curve that has tangency order two with the fibre $\fibre{\pi}{p}$: the claim is local, so we can consider that $C$ is given by a parametrisation
$$
n\colon \Delta \longrightarrow \C^2, \ t \ \longmapsto \ (t^2, t^3).
$$
Denote by $x,y$ the coordinates on $\C^2$; then we have induced coordinates 
$(x,y), [u_x, u_y]$ on 
$$
\PP(\Omega_S) \simeq \mathbf P(T_S) \simeq \C^2 \times \mathbf P \left(\C \frac{\partial}{\partial x} \oplus 
\C \frac{\partial}{\partial y}\right).
$$
The tangent map $T n$ is given by
$$
t \longmapsto 2t \frac{\partial}{\partial x} + 3 t^2 \frac{\partial}{\partial y},
$$
so we see that $Q_C \simeq \sI_0 \otimes \omega_{\Delta}$ and $\tilde n$ is given by
$$
t \longmapsto \left(
(t^2, t^3), \left[2 \frac{\partial}{\partial x} + 3 t \frac{\partial}{\partial y}\right]
\right).
$$ 
This map is well defined in the origin and maps it onto the point $ \left((0,0), [2:0]\right)$. In the affine chart $u_x \neq 0$ the map $\tilde n$ is given by
$$
t \longmapsto (t^2, t^3, 3 t).
$$
In the coordinates $(x,y,z)$ of the affine chart, the image of this map is the smooth curve cut out by $9x=z^2, 27y=z^3$. The intersection of this curve with the fibre $x=0, y=0$ is a finite non-reduced scheme of length two.
\end{enumerate}
\end{remarks}

\begin{notation} \label{notation-node-cusp}
We will denote by $\Cnode$ (resp.\ $\Ccusp$) an element of $|L|$ that has exactly two nodes (resp.\ a cuspidal point).\footnote{This notation is justified by the fact that these curves correspond to the nodes (resp.\ cusps) of the curve $R^\vee$; see Section~\ref{subsection-surface-D}.}

We will denote by $\CnodeS$ (resp.\ $C_{\cusp, S}$) the canonical lifting of the curve $\Cnode$ (resp.\ $C_{\cusp}$) to $\PP(\Omega_S)$.
\end{notation}

\begin{lemma} \label{lemma-D-DS}
In the situation summarised by Diagram~\ref{diagram}, let $C \subset \PP(f^* \Omega_{\PP^2})$ be a fibre of the universal family $\PP(f^* \Omega_{\PP^2}) \rightarrow |L|$, and let $C_S \subset \PP(\Omega_S)$ be the canonical lifting of the curve $[C] \in |L|$. Then $C_S$ is the strict transform of\, $C$ under the birational map $\mu_S \circ \mu_P^{-1}$.

In particular let
$$
D_S := \bigcup_{[C] \in R^\vee} C_S \subset \PP(\Omega_S)
$$ 
be the irreducible surface obtained by canonical liftings of singular elements of\, $|L|$.  Then we have $D_S = \mu_S(D)$.
\end{lemma}

\begin{proof}
Since $D$ is the strict transform of $\fibre{q}{R^\vee}$, which is the universal family of singular elements of $|L|$, the second statement follows from the first.

For the proof of the first statement, recall (\textit{cf.} the proof of Lemma~\ref{lemma-singularities-D}) that the fibres of $\holom{p_2}{\PP(\Omega_{\PP^2})}{(\PP^2)^\vee}$ are given by liftings corresponding to canonical quotients $\Omega_{\PP^2} \otimes \sO_d \rightarrow \omega_d$ of lines $ d \subset \PP^2$. Since $q$ is obtained by base change from $p_2$, we see that the fibres of $q$ are given by liftings corresponding to quotients
$$
f^* \Omega_{\PP^2} \otimes \sO_C \longrightarrow (f|_C)^* \omega_d.
$$
In the complement of the ramification divisor, the vector bundle $f^* \Omega_{\PP^2} \otimes \sO_C$ (resp.\ $(f|_C)^* \omega_d$) coincides with $\Omega_S \otimes \sO_C$ (resp.\ $\omega_C$). Thus the two liftings coincide in the complement of $R$; hence the curves are strict transforms of each other.
\end{proof}

\subsection{Geometry of the surface $\boldsymbol{\bar D}$}
\label{subsection-geometry-barD}

This subsection contains the technical core of our study.  We will use the notation introduced in the Sections~\ref{subsection-elementary-transform} and \ref{subsection-surface-D}, in particular Diagram \eqref{diagram}.  We still denote by
$$
\holomd{p_1}{\PP(\Omega_{\PP^2})}{\PP^2}, \qquad \holomd{p_2}{\PP(\Omega_{\PP^2})}{(\PP^2)^\vee}
$$
the projections and by $H_1$ (resp.\ $H_2$) the hyperplane classes on $\PP^2$ (resp.\ $(\PP^2)^\vee$). In order to simplify the notation, we will identify
$$
(\PP^2)^\vee = |L|, \qquad R=B, \qquad R^\vee=B^\vee.
$$
Restricting the universal family $\holom{q}{\PP(f^* \Omega_{\PP^2})}{(\PP^2)^\vee}$ over the dual curve, we obtain a flat fibration $\holom{q|_{\fibre{q}{R^\vee}}}{\fibre{q}{R^\vee}}{R^\vee}$. The singularities of the source and target make it difficult to analyse this fibration. Therefore we will normalise both spaces to obtain a fibration
$$
\holomd{\bar q}{\bar D}{R}
$$
and show the following. 

\begin{theorem} \label{theorem-structure-barD}
Let $\holom{\bar \nu}{\bar D}{\fibre{q}{R^\vee}}$ be the normalisation. Then $\bar D$ is a smooth minimal projective surface, and the elliptic fibration $\holom{\bar q}{\bar D}{R}$ induced by $q|_{\fibre{q}{R^\vee}}$ has exactly 648 singular fibres which are all of Kodaira type $I_1$ $($\textit{i.e.}\ nodal cubics$)$.
\end{theorem}

\begin{remark} \label{remark-dedieu}
Let $R \rightarrow R^{\vee}$ be the normalisation. The fibre product 
$$
(\fibre{q}{R^\vee} \times_{R^{\vee}} R) \longrightarrow R
$$
is a family of curves that satisfies the assumptions of Teissier's simultaneous normalisation theorem \cite[Section~I.1.3.2, Theorem~1]{Tes77}. Thus the fibration $\bar D \rightarrow R$ is smooth near the bisection determined by the preimage of the rational section $R_P \subset \fibre{q}{R^\vee}$. Note that Teissier's theorem does not imply the smoothness of $\bar D$ in the 648 remaining nodes.  
\end{remark}

The following commutative diagram will guide the reader through the construction. The varieties and morphisms in columns 3 to 5 have been introduced in the Sections~\ref{subsection-elementary-transform} and~\ref{subsection-surface-D}. The second column is obtained from the third column by normalisation. The curves in the first column will be successively introduced in this subsection.

\begin{equation} \label{bigdiagram}
\xymatrix{
N_{\tilde D} \ar[d]  \ar @{^{(}->}[r] & \tilde D \ar[r]^{\nu} \ar[d]^{\tilde \mu_P} & D  \ar[d]^{\mu_P|_D} \ar @{^{(}->}[r] & Y \ar[d]^{\mu_P}  &
\\
N_{\bar D}  \ar[d]  \ar @{^{(}->}[r] & \bar D \ar[d]^{\tilde f_{\bar D}} \ar[r]^{\bar \nu} & \fibre{q}{R^\vee}  \ar[d] \ar @{^{(}->}[r] & \PP(f^* \Omega_{\PP^2}) \ar[d]^{\tilde f} \ar[r]^p & S \ar[d]^f
\\
R_T  \ar @{^{(}->}[r] &
T \ar[d]^{q_T} \ar[r]^{\nu_T} & \fibre{p_2}{R^\vee} \ar[d]  \ar @{^{(}->}[r] & \PP(\Omega_{\PP^2})  \ar[d]^{p_2} \ar[r]^{p_1} & \PP^2
\\
& R \ar @/^/[lu] \ar[r]^{n_R} & R^\vee \ar @{^{(}->}[r] &  (\PP^2)^\vee\rlap{.} &
}
\end{equation}

Let 
$$
R_{\PP^2} \subset \PP(\Omega_{\PP^2})
$$
be the curve defined by the canonical quotient $\Omega_{\PP^2} \otimes \sO_R \rightarrow \omega_R$ (here we use the identification $B=R$). The
curve $R_{\PP^2}$ maps birationally onto $R$ (resp.\ $R^\vee$), so its class in $\PP(\Omega_{\PP^2})$ is
 \begin{equation} \label{formulaRP}
R_{\PP^2}  = 30 l_1 + 6 l_2, 
\end{equation}
where $l_i$ is a fibre of the fibration $p_i$. 
Let $\holom{n_R}{R}{R^\vee}$ be the normalisation, and let 
$$
\holomd{\nu_T}{T}{\fibre{p_2}{R^\vee} \subset \PP(\Omega_{\PP^2})}
$$ be the normalisation of $\fibre{p_2}{R^\vee}$.
Since $\fibre{p_2}{R^\vee} \rightarrow R^\vee$ is locally trivial with fibre $\PP^1$, we obtain a ruled surface
$$
\holomd{q_T}{T}{R}.
$$
In fact we have $T \simeq \PP(n_R^* \Omega_{(\PP^2)^\vee})$, and we set
$$
\zeta_T := \nu_T^* c_1(\sO_{ \PP(\Omega_{(\PP^2)^\vee}) }(1))
$$
for the tautological class.  The curve $R_{\PP^2}\subset \fibre{p_2}{R^\vee}$ is not contained in the singular locus of $\fibre{p_2}{R^\vee}$, and its strict transform is a section $R_T \subset T$.

\begin{lemma} \label{lemmasurfaceT}
In the situation summarised by Diagram~\ref{bigdiagram}, we have
$$
R_T^2 = -18 < 0.
$$
Thus $R_T \subset T$ is the unique curve with negative self-intersection, and we have
$$
\NE{T} = \langle R_T, l_T \rangle,
\qquad
\mbox{\rm Nef}(T) =  \langle R_T+18 l_T, l_T \rangle, 
$$
where $l_T$ is a fibre of $q_T$.
Moreover we have 
$$
K_T \equiv - 2 \zeta_T  - 72 l_T. 
$$
\end{lemma}

\begin{proof}
Observe that $\PP(\Omega_{\PP^2}) \simeq \PP(\Omega_{(\PP^2)^\vee})$ and
$c_1(\sO_{ \PP(\Omega_{(\PP^2)^\vee}) }(1)) = p_1^* H_1 - 2 p_2^* H_2$.
Using \eqref{formulaRP} we obtain
$$
c_1(\sO_{ \PP(\Omega_{(\PP^2)^\vee}) }(1))  \cdot R_{\PP^2}= - 54.
$$
Since $R^\vee \equiv 30 H_2$, we have
$$
c_1(\sO_{ \PP(\Omega_{(\PP^2)^\vee}) }(1))^2 \cdot \fibre{p_2}{R^\vee} = -90.
$$ 
Since $\zeta_T \equiv \nu_T^* c_1(\sO_{ \PP(\Omega_{(\PP^2)^\vee}) }(1))$ and $R_T \rightarrow
R_{\PP^2}$ is birational, this implies that
\begin{equation} \label{zetasquare}
\zeta_T^2 = - 90, \qquad \zeta_T \cdot R_T = -54.
\end{equation}
Since $R_T$ is a $q_T$-section, we conclude that
\begin{equation} \label{RTclass}
R_T \equiv \zeta_T + 36 l_T
\end{equation}
and hence $R_T^2=-18$.
The description of the Mori cone and the nef cone is now standard; \textit{cf.}~\cite[Section~V.2]{Har77}.
The canonical class of $T$ is given by 
$$
K_T \simeq q_T^* (K_R + c_1(n_R^* \Omega_{(\PP^2)^\vee})) - 2 \zeta_T. 
$$
Since $\deg K_R = 18$ and 
$$
\deg (n_R^* \Omega_{(\PP^2)^\vee}) = \deg (\Omega_{(\PP^2)^\vee} \otimes \sO_{R^\vee})
= - 3 H_2 \cdot 30 H_2 = -90,  
$$
we get $K_T \equiv - 2 \zeta_T  - 72 l_T$.
\end{proof}

\begin{lemma} \label{lemmaclassBT}
In the situation summarised by Diagram~\ref{bigdiagram},
let $\holom{\bar \nu}{\bar D}{\fibre{q}{R^\vee}}$ be the normalisation. By the universal property of the normalisation, we have an induced two-to-one cover
$$
\holomd{\tilde f_{\bar D}}{\bar D}{T}, 
$$
and we denote the branch locus by $B_T \subset T$.
Then $B_T$ does not contain any fibre of the ruling, and its class is
$$
B_T \equiv 4 \zeta_T + 288 l_T.
$$
Moreover, one has
$$
K_{\bar D} \equiv 72 \tilde f_{\bar D}^* l_T. 
$$
\end{lemma}

Note that at this point, we do not know that $\bar D$ is smooth. However, being a cyclic double cover of a smooth surface, it is Gorenstein.

\begin{proof}
The double cover $\fibre{q}{R^\vee} \rightarrow \fibre{p_2}{R^\vee}$ 
is branched over the scheme $p_1^* R \cap \fibre{p_2}{R^\vee}$ which 
does not contain any fibre of $\fibre{p_2}{R^\vee} \rightarrow R^\vee$.
Thus, since $\tilde f_{\bar D}$ is obtained from $\fibre{q}{R^\vee} \rightarrow \fibre{p_2}{R^\vee}$ by normalisation, its branch locus does not contain a fibre of the ruling.

Since $R \equiv 6 H_1$, by \eqref{formulaRP} we have 
$$
\nu_T^* p_1^* R \cdot R_T = p_1^* R \cdot R_{\PP^2}= 36.
$$
Also note that $\nu_T^* p_1^* R$ has degree six on the fibres of $q_T$ since 
a fibre maps onto a line in $\PP^2$.
Thus using \eqref{zetasquare}, we deduce that
$$
\nu_T^* p_1^* R \equiv 6 \zeta_T + 360 l_T.
$$
If $l_T$ is a general $q_T$-fibre, we can identify it to the corresponding $p_2$-fibre. Hence  we know that this fibre intersects $p_1^* R$ transversally in four points and with multiplicity two in the point which will give the node. This last point is on the curve $R_{\PP^2}$, so we see that
the effective divisor  $\nu_T^* p_1^* R$ contains its strict transform $R_T$ with multiplicity two.
Thus we can write
\begin{equation} \label{decomposeR}
\nu_T^* p_1^* R = 2 R_T + B_T; 
\end{equation}
hence by \eqref{RTclass}, we have $B_T \equiv 4 \zeta_T + 288 l_T$.
Since $\fibre{q}{R^\vee} \rightarrow \fibre{p_2}{R^\vee}$  ramifies exactly over $p_1^* R \cap \fibre{p_2}{R^\vee}$ and the surface $\fibre{q}{R^\vee}$ has a nodal singularity in the generic point of $R_P \simeq R_{\PP^2}\simeq R_T$ (\textit{cf.} Lemma~\ref{lemma-singularities-D}), the branch locus of  $\tilde f_{\bar D}$ is exactly $B_T$.

Now we can compute the canonical class of $\bar D$: since $\tilde f_{\bar D}$ is a double cover ramified along $B_T$, we have $K_{\bar D} \equiv \tilde f_{\bar D}^* K_T + \frac{1}{2} \tilde f_{\bar D}^* B_T$. Thus we obtain
\[\pushQED{\qed} 
K_{\bar D} \equiv \tilde f_{\bar D}^* (
- 2 \zeta_T  - 72 l_T 
+ \tfrac{1}{2}
(
4 \zeta_T + 288 l_T
)
) 
= 72 \tilde f_{\bar D}^* l_T. \qedhere
\popQED
\]
\renewcommand{\qed}{}   
\end{proof}
%
%\begin{remark*} {\color{teal} I found the mistake in Lemma~\ref{lemmaclassBT} since it was not consistent with my RR computation on $\bar D$:}
%We can check that these formulas are correct (assuming $\bar D$ smooth):
%the elliptic fibratin $\bar D \rightarrow R$ has exactly 648 nodal fibres, so by 
%\cite[III, Prop.11.4]{BHPV04} we have $e(\bar D)=648$. Using Noether's formula
%we obtain that $\chi(\bar D)=54$. Since $q(\bar D)=q(R)=10$ we see that
%$p_g(\bar D)=63$. Since by Lemma~\ref{lemmaclassBT} we have $K_{\bar D} \equiv 72 fibres$,
%we see that 
%$$
%p_g(\bar D) = h^0 (R, G)
%$$ 
%with $G$ a divisor of degree $72$ on $R$. Since $72>\deg K_R$ we compute by Riemann-Roch
%that $h^0(R, G) = \chi(R, G) = 1 - 10 + 72=63$, as expected.
%\end{remark*}

We will now analyse the singularities of $\bar D$ via the double cover $\tilde f_{\bar D}$: the morphism
$\tilde f_{\bar D}$ is the degree two cyclic cover associated to a line bundle isomorphic to
$\frac{1}{2} B_T \equiv 2 \zeta_T + 144 l_T$, so by \cite[Section~4.1.B]{Laz04a}, given a point $p \in T$ and a local equation
$g(x,y)$ for $B_T$ near $p$, a local equation of $\bar D$ is 
$$
z^2 = g(x,y).
$$
In particular $\bar D$ is singular if and only if
$B_T$ is singular in $p$.

\begin{proof}[Proof of Theorem~\ref{theorem-structure-barD}]
By Lemma~\ref{lemmaclassBT} we know that the canonical class is nef, so the surface is minimal. We will show that the branch locus $B_T$ is smooth; hence $\bar D$ is smooth.

Let $\holom{n_R}{R}{R^\vee}$ be the normalisation.
Fix a point $r \in R$, and denote by $T_r$ the $q_T$-fibre over the point~$r$.
The double cover $\fibre{\tilde f_d}{T_r} \rightarrow T_r$ is determined by the double cover $f|_C\colon C \rightarrow d$
of the line $d \subset \PP^2$ corresponding to the point $n_R(r)$.
 We make a case distinction:
\begin{itemize}
\item If $n_R(r)$ is smooth, the corresponding cover $C \rightarrow d$ ramifies
in the four points where $d$ is not tangent to the branch divisor $B$. Since by Lemma~\ref{lemmaclassBT}, the curve $B_T$ does not contain a fibre of the ruling and has degree four on every fibre, 
we see that $B_T \cap T_r$ is smooth. Thus $B_T$ is smooth near $T_r$.
\item If $n_R(r)$ is a cusp, the corresponding cover $C \rightarrow d$ ramifies
in the three points where $d$ is not tangent to $B$. Thus the finite scheme
$B_T \cap T_r$ has at least three smooth points. 
Since the scheme has length four, it is smooth. Thus $B_T$ is smooth near $T_r$.
\item If $n_R(r)$ is a node, the corresponding cover $C \rightarrow d$ ramifies in the two points where $d$ is not tangent to $B$.  Since the normalisation $\tilde C \rightarrow C$ is a rational curve, the double cover $\tilde C \rightarrow l$ has exactly two ramification points.  This shows that $B_T \cap T_r$ has two smooth points and a point $t \in T_r$ with multiplicity two supported on one of the two points where $d$ is tangent to $B$ (the other point is on $R_T$, whence the asymmetry). Thus $B_T \cap T_r$ is not smooth; nevertheless, we claim that $B_T$ is smooth near $T_r$: the fibration $p_2$ is the universal family of lines on $\PP^2$, so by construction the line $p_1(\fibre{p_2}{n_R(r)})$ is tangent to $R$. Thus $\fibre{p_2}{n_R(r)}$ is tangent to $p_1^* R$.  Since $B_T$ is the unique component of $\nu_T^* p_1^* R$ passing through $t$, we obtain that $B_T$ is tangent to the fibre $T_r$. Since the local intersection number in the point $t$ is two, we deduce (\textit{e.g.} by \cite[p.225,
  Axiom 5)]{Per08}) that $B_T$ is smooth near this point.
\end{itemize}

The analysis of the branch locus also allows us to determine the singular fibres of $\holom{\bar q}{\bar D}{R}$: since $q_T$ is a smooth fibration, a fibre of $\bar q=q_T \circ \tilde f_{\bar D}$ is singular if and only if the restriction of the branch locus $B_T$ to the fibre $T_r$ is singular.  We have seen that this happens if and only if $n_R(r)$ is a node. Since $R^\vee$ has 324 nodes, we see that there are 648 singular fibres. Since the intersection $B_T \cap T_r$ has two reduced points and one point with multiplicity two, the double cover is a nodal cubic.
\end{proof}

For later use, we note the following. 

\begin{corollary} \label{corollary-BT-irreducible}
The curve $B_T \subset T$ is smooth and irreducible.
\end{corollary}

\begin{proof}
The smoothness of $B_T$ was shown in the proof of Theorem~\ref{theorem-structure-barD}, so we are left to show that $B_T$ is connected. Since $B_T$ does not contain the curve $R_T$, we know by Lemma~\ref{lemmasurfaceT} that the irreducible components of $B_T$ are nef divisors. Since $B_T=4 \zeta_T+288 l_T$ by Lemma~\ref{lemmaclassBT}, we can use \eqref{zetasquare} to compute that $B_T^2=864>0$. Thus $B_T$ is a nef and big divisor, hence connected.
\end{proof}

\subsection{Geometry of the surface $\boldsymbol{\tilde D}$}
\label{subsection-geometry-tildeD}

The first goal of this subsection is to describe the normalisation of the surface $D$. 

\begin{theorem} \label{theorem-structure-D}
Let $\holom{\nu}{\tilde D}{D}$ be the normalisation. Then $\tilde D$ is a smooth projective surface that is the blow-up of $\bar D$ in 720 points: the 648 nodal points of the singular fibres and 72 points on the elliptic fibres which correspond to cuspidal curves in $|L|$.
\end{theorem}

\begin{proof}
For simplicity's sake we denote by 
$$
\holomd{q}{\fibre{q}{R^\vee}}{R^\vee}, \qquad \holomd{\psi}{D}{R^\vee}
$$
the restriction of the fibration $q$ over the curve $R^\vee$ (resp.\ the restriction of $\psi$ to $D$). Given a point $r \in R^\vee$ we will describe
the map $D \rightarrow \fibre{q}{R^\vee}$ in a neighbourhood of the fibre $\fibre{q}{r}$. This local description will then allow us to show that $\tilde D$ is smooth. 

{\em Case~1: $r \in R^\vee$ is a smooth point.} We know by Lemma~\ref{lemma-singularities-D} that $\fibre{q}{R^\vee}$ has nodal singularities near $\fibre{q}{r}$, so the blow-up along $R_P$ coincides with the normalisation, and $D$ is smooth near $\fibre{\psi}{r}$.

{\em Case~2: $r \in R^\vee$ is a node.}
Denote by $\Delta \subset R^\vee$ and $\Delta' \subset R^\vee$  the two local
branches through $r$. Then $\Delta$ is smooth, and we can suppose that it is a small disc that contains no other singular points of $R^\vee$. The preimage
$$
\fibre{q}{\Delta} \subset \PP(f^* \Omega_{\PP^2})
$$
is an analytic hypersurface, so Gorenstein, and generically reduced, hence reduced. 
We have seen in Section~\ref{subsection-surface-D} that the fibres of 
$$
\fibre{q}{\Delta} \longrightarrow \Delta
$$
for $t \neq 0$ are curves with exactly one node and the central fibre has exactly two nodes.  Denote by $x_1, x_2$ the nodes of the central fibre and by $x_t$ the unique node of the fibre over $t \neq 0$. Then up to renumbering,  $x_1$ is in the closure of $\cup_{t \neq 0} x_t$, so these points form a curve $R_\Delta$ that is a section of
$$
\fibre{q}{\Delta} \longrightarrow \Delta, 
$$
and the curve $R_\Delta$ is in the non-normal locus of $\fibre{q}{\Delta}$. Since $R_\Delta$ is a section over a smooth curve, the point $x_2$ is not in $R_\Delta$.  Since $\Delta$ is smooth and $\fibre{q}{\Delta} \rightarrow \Delta$ has smooth fibres in the complement of $R_\Delta \cup x_2$, we see that the singular locus of $\fibre{q}{\Delta}$ is contained in $R_\Delta \cup x_2$. Since $x_2$ is an isolated singularity and $\fibre{q}{\Delta}$ is Gorenstein, we see that $x_2$ is in the normal locus of $\fibre{q}{\Delta}$.  Let $\widetilde{\fibre{q}{\Delta}} \rightarrow \fibre{q}{\Delta}$ be the normalisation. By the universal property of the normalisation, we have an embedding $\widetilde{\fibre{q}{\Delta}} \hookrightarrow \bar D$. Since $\bar D$ is smooth by Theorem~\ref{theorem-structure-barD}, we see that $\widetilde{\fibre{q}{\Delta}}$ is smooth; in particular  $x_2$ is a smooth point $\fibre{q}{\Delta}$.

We can now describe the blow-up of $\fibre{q}{\Delta}$ along the scheme $\fibre{q}{\Delta} \cap R_P$. By construction $R_P$ contains all the nodes over $R^\vee$, so we see that set-theoretically, 
$$
R_P \cap \fibre{q}{\Delta} = R_\Delta \cup x_2.
$$
Since $\fibre{q}{\Delta}$ has nodal singularities along $R_\Delta \subset \fibre{q}{\Delta} \cap R_P$, the blow-up coincides with the normalisation.  We also observe that the curve $R_P$ meets $\fibre{q}{\Delta}$ transversally in $x_2$: if the intersection is not transversal, the intersection of the images $q(\fibre{q}{\Delta})$ and $q(R_{\Delta'})$ is not transversal. But $q(\fibre{q}{\Delta})=\Delta$ and $q(R_{\Delta'})=\Delta'$, so they intersect transversally.
Thus near the point $x_2$, the blow-up of $\fibre{q}{\Delta}$ along the scheme $R_P \cap \fibre{q}{\Delta}$ coincides with the blow-up of the reduced point $x_2 \in \fibre{q}{\Delta}$; hence the blow-up is smooth.

We can now conclude as follows: near the fibre $\fibre{q}{r}$ the surface $\fibre{q}{R^\vee}$ is reducible with irreducible components $\fibre{q}{\Delta}$ and $\fibre{q}{\Delta'}$. The strict transform of $\fibre{q}{R^\vee}$ in $Y$ coincides with the union of the strict transforms of $\fibre{q}{\Delta}$ and $\fibre{q}{\Delta'}$ in $Y$. We have seen that the strict transforms of $\fibre{q}{\Delta}$ and $\fibre{q}{\Delta'}$ are smooth. Thus in a neighbourhood of $\fibre{\psi}{r}$, the normalisation $\tilde D \rightarrow D$ consists just of separating the two irreducible components.

This finishes the discussion of this case. Note that the argument above also shows that near one of the 648 nodal fibres of the elliptic fibration $\holom{\bar q}{\bar D}{R}$ (\textit{cf.}~Theorem~\ref{theorem-structure-barD}), the birational map $\tilde D \rightarrow D$ is given by the blow-up of the node of the fibre.

{\em Case 3: $r \in R^\vee$ is a cusp.}
Let $x \in \fibre{q}{r}$ be the unique singular point. The statement is clear in the complement of $x$, so a local computation near $x$ will allow us to conclude:
the universal family of lines $\PP(T_{\PP^2}) \subset \PP^2 \times (\PP^2)^\vee$
can be given by the equation
$$
\sum_{i=0}^2 x_i y_i = 0,
$$
where the $x_i$ are coordinates on $\PP^2$ and $y_i$ on $(\PP^2)^\vee$.  Identifying $x \in \fibre{q}{r} \subset S$ to its image in $S$, we can assume that $f(x)=[1:0:0]$ and work in the affine chart $x_0=1$ on $\PP^2$.  Since the branch curve $B$ has a simple inflection line in $(0,0)$, we can suppose that it is given in the local chart by $x_2 = x_1^3$.  The tangent line to $B$ in the origin is given by $x_2=0$, so it corresponds to the point $[0:0:1] \in (\PP^2)^\vee$. Therefore, we choose the affine chart $y_2=1$ for $(\PP^2)^\vee$, so the universal family of lines is given by
$$
y_0 + x_1 y_1 + x_2 = 0.
$$
We choose local coordinates $(u_1,u_2)$ on $S$ near the point $x$ such that the double cover $f$ is given by 
$$
(u_1, u_2) \longmapsto (u_2, u_1^2+u_2^3).
$$
The coordinates $(u_1, u_2, y_0, y_1)$ are thus local coordinates on $S \times (\PP^2)^\vee$. In these coordinates the divisor $\PP(f^* \Omega_{\PP^2})$ is given by the equation
$$
y_0 + u_2 y_1 + u_1^2+u_2^3 = 0.
$$
It is not difficult to check that the local equation of $R^\vee$ is
$$
4 y_1^3 + 27 y_0^2 = 0,
$$
so the surface $\fibre{q}{R^\vee} \subset \PP(f^* \Omega_{\PP^2})$ is given by the local equations
$$
y_0 + u_2 y_1 + u_1^2+u_2^3 = 0, \qquad
4 y_1^3 + 27 y_0^2 = 0.
$$
The curve $R_P \subset \PP(f^* \Omega_{\PP^2})$ can be parametrised by 
$$
t \longmapsto (0, t, 2 t^3, -3 t^2),
$$
so the map
$$
\C^4 \longrightarrow \C^4, \quad
(u,t,v,w)  \longmapsto (u,t, 2t^3+uv, -3t^2+uw)
$$
is a local chart\footnote{The computation for the other charts is analogous, but simpler.} for the blow-up of $S \times (\PP^2)^\vee$ along the curve $R_P$.
In this local chart  the equation of the threefold $Y \subset \mbox{Bl}_{R_P} S \times (\PP^2)^\vee$
is
$$
v+tw+u=0,
$$ 
while the strict transform of the hypersurface $4 y_1^3 + 27 y_0^2 = 0$
is given by
$$
108 t^4 w - 36 t^2 u w^2 + 4 u^2 w^3 + 108 t^3 v + 27 u v^2 = 0.
$$
Substituting $v=-tw-u$ in the last equation, we obtain local equations of $D$
in $\mbox{Bl}_{R_P} S \times (\PP^2)^\vee$:
\begin{equation} \label{equation-D}
v+tw+u=0, \qquad -9t^2w^2 + 4 u w^3 - 108 t^3 + 54 t w u + 27 u^2 = 0.
\end{equation}
Let $C_{\cusp, Y}$ be the strict transform of the cuspidal curve $\fibre{q}{r}$.  Computing the Jacobian matrix for the system of equations \eqref{equation-D}, we see that $D$ is singular along $C_{\cusp, Y}$, but its unique singular point on the exceptional divisor $u=0$ is the origin.  Since the origin is contained in $C_{\cusp, Y}$, it is a non-normal point of~$D$.  Thus we are done if we show that the normalisation is smooth near the preimage of $C_{\cusp, Y}$: note that $C_{\cusp, Y}$ can locally be parametrised by
$$
w \mapsto \left(\frac{w^3}{27}, -\frac{w^2}{9}, \frac{2w^3}{27},  w\right),
$$
so
$$
\C^4 \longrightarrow \C^4, \quad
(\alpha, \beta, \gamma, \delta) \longmapsto \left(\alpha \beta+\frac{\delta^3}{27}, \beta- \frac{\delta^2}{9}, \beta \gamma+\frac{2 \delta^3}{27},  \delta\right)
$$
gives a local chart for the blow-up of $\C^4$ along this curve (the other charts are less interesting for the proof and left to the reader).
The strict transform of the hypersurface $v+tw+u=0$ under this blow-up
is given by
$$
\alpha+\delta+\gamma=0,
$$
and the strict transform of $-9t^2w^2 + 4 u w^3 - 108 t^3 + 54 t w u + 27 u^2 = 0$ is
$$
-108 \beta + 27 \delta^2 + 54 \alpha \delta + 27 \alpha^2 = 0.
$$
Hence the strict transform $D'$ of the surface $D$ is given by
$$
\alpha+\delta+\gamma=0,  \qquad -108 \beta + 27 \delta^2 + 54 \alpha \delta + 27 \alpha^2 = 0.
$$
It is straightforward to check that $D'$ is smooth and the morphism $D' \rightarrow D$ is finite. By the universal property of the normalisation $\tilde D \rightarrow D$, we obtain an embedding $D' \hookrightarrow \tilde D$, so $\tilde D$ is smooth.
\end{proof}

We are finally ready to prove our first main result. 

\begin{proof}[Proof of Theorem~\ref{theorem-main1}]
We have shown in Lemma~\ref{lemma-D-DS} that $D_S = \mu_S(D)$. We claim that the birational morphism $D \rightarrow D_S$ is finite. In particular the normalisation of $D_S$ is given by the surface $\tilde D$, so the statement follows from Theorem~\ref{theorem-structure-D}.

For the proof of the claim, recall that $D$ is the strict transform of $\fibre{q}{R^\vee} \subset \PP(f^* \Omega_{\PP^2})$ and the curves $l_S$ contracted by $\mu_S$ are the strict transforms of the fibres of $\PP(f^* \Omega_{\PP^2} \otimes \sO_R) \rightarrow R_S$ (\textit{cf.} Corollary~\ref{corollarypseffES}). Thus it is sufficient that $\fibre{q}{R^\vee}$ does not contain any $p$-fibres; yet this is clear: a $p$-fibre is mapped by $q$ onto a line in $(\PP^2)^\vee$, so it is not contained in the irreducible curve $R^\vee$.

In order to determine the class of $D_S$, 
it is sufficient to compute the push-forward of the class
$D \equiv 30M-2E_P$ (\textit{cf.} Proposition~\ref{proposition-class-D}).

Since $\mu_S^* (\zeta_S+\pi^* 2L) \equiv M + E_S$ (\textit{cf.} Remark~\ref{remark-zeta-twoL}), we have $(\mu_S)_* M \equiv \zeta_S+\pi^* 2L$. We saw in Section~\ref{subsection-elementary-transform} that $\mu_S(E_P) = \fibre{\pi}{R}$. Since $R \simeq 3L$ we obtain
\[\pushQED{\qed} 
(\mu_S)_* D \equiv (\mu_S)_* (30 M - 2 E_P) \equiv 30 \zeta_S + \pi^* 54 L. \qedhere
\popQED
\] \renewcommand{\qed}{}
\end{proof}

Theorem~\ref{theorem-structure-D} gives a complete description of the elliptic fibration $\tilde D \rightarrow R$. In order to determine an intersection table for $\tilde D$, we are left to describe a curve that is horizontal with respect to the fibration.

\begin{lemma} \label{lemma-NDbar}
Let $\tilde f_{\bar D}\colon \bar D \rightarrow T$ be the double cover introduced in Lemma~\ref{lemmaclassBT}, and let $\holom{\bar q}{\bar D}{R}$ be the elliptic fibration $($\textit{cf.}\ Theorem~\ref{theorem-structure-barD}\,$)$.  Set
$$
N_{\bar D} := \tilde f_{\bar D}^* R_T.
$$
Then $N_{\bar D}$ is a smooth irreducible curve that is a bisection of $\bar q$, and
\begin{equation} \label{NDbarsquare}
N_{\bar D}^2 = -36.
\end{equation}
Moreover 
\begin{itemize}
\item if $t \in R$ is a point mapping onto a node in $R^\vee$, the intersection of $N_{\bar D}$ with the fibre $\fibre{\bar q}{t}$ consists of two smooth points; in particular,  $N_{\bar D}$ is disjoint from the node in $\fibre{\bar q}{t}$; 
\item if $t \in R$ is a point mapping onto a cusp in $R^\vee$, the intersection of $N_{\bar D}$ with the fibre $\fibre{\bar q}{t}$ is non-reduced, its support being the unique point of $\fibre{\bar q}{t}$ mapping onto the cusp $\fibre{q}{\nu(t)}$.
\end{itemize}
\end{lemma}

\begin{proof}
We know by \eqref{decomposeR} that the pull-back of $p_1^* R$ to $R_T$ decomposes as $2 R_T+B_T$, where $B_T$ is the branch divisor of $\tilde f_{\bar D}$. In particular $\tilde f_{\bar D}$ is \'etale in the generic point of $R_T$, so the pull-back $\tilde f_{\bar D}^* R_T$ is a reduced bisection, and by Lemma~\ref{lemmasurfaceT} one has
$$
N_{\bar D}^2 = (\tilde f_{\bar D}^* R_T)^2 = 2 \cdot (-18) = -36.
$$
In order to see that $N_{\bar D}$ is smooth, we observe, \textit{e.g.} by using the description in the proof of Theorem~\ref{theorem-structure-barD}, that $R_T$ is disjoint from the branch divisor $B_T$ in the complement of the fibres corresponding to cusps. Thus its preimage $N_{\bar D}$ is smooth in the complement of the fibres corresponding to cusps.  For a fibre over a cusp, the section $R_T$ passes exactly through the inflection point.  Thus, analogously to the proof of Theorem~\ref{theorem-structure-D}, a local computation shows that $N_{\bar D}$ is also smooth near these fibres.

Finally observe that for a point $t \in R$ mapping onto a node in $R^\vee$, the curve $R_T$ is disjoint from the branch curve $B_T$, so there are two distinct points $N_{\bar D}$ mapping onto $R_T \cap \fibre{q_T}{t}$. For a point $t \in R$ mapping onto a cusp in $R^\vee$, the intersection $R_T \cap \fibre{q_T}{t}$ is contained in $B_T$, so the set-theoretic preimage is a single point.
\end{proof}

\begin{proposition} \label{proposition-intersection-D}
In the situation of Theorem~\ref{theorem-structure-D}, let $N_{\tilde D}$ be the strict transform of $N_{\bar D}$ $($\textit{cf.} Lemma~\ref{lemma-NDbar} above$)$ under the birational map $\holom{\tilde \mu_P}{\tilde D}{\bar D}$.  Moreover
\begin{itemize}    
\item if $t \in R$ is a point mapping onto a node in $R^\vee$, denote by $C_{\node,t} \subset \tilde D$ the strict transform of the nodal cubic $\fibre{q}{t}$ and by $l_{\node,t} \subset \tilde D$ the exceptional curve mapping onto the node;
\item if $t \in R$ is a point mapping onto a cusp in $R^\vee$, denote by $C_{\cusp,t} \subset \tilde D$ the strict transform of the smooth elliptic curve $\fibre{q}{t}$ and by $l_{\cusp,t}$ the exceptional curve mapping onto a point in $\fibre{q}{t}$.
\end{itemize}
Then we have the following intersection numbers:
\begin{gather*}
C_{\node,t} \cdot l_{\node,t} = 2, \ \ \  C_{\node,t}^2=-4, \ \ \ l_{\node,t}^2 = -1,\\
C_{\cusp,t} \cdot l_{\cusp,t} = 1, \ \ \ C_{\cusp,t}^2=-1, \ \ \ l_{\cusp,t}^2=-1,\\
N_{\tilde D} \cdot l_{\node,t}=0, \ \ \ N_{\tilde D} \cdot C_{\node,t} = 2, \ \ \ 
N_{\tilde D} \cdot l_{\cusp,t}=1, \ N_{\tilde D} \cdot C_{\cusp,t} = 1,\\
N_{\tilde D}^2 = -108.
\end{gather*}
\end{proposition}

\begin{proof}
The intersection numbers that do not involve $N_{\tilde D}$ are immediate consequences of Theorem~\ref{theorem-structure-D}. The intersection numbers involving $N_{\tilde D}$ follow from Lemma~\ref{lemma-NDbar}.
\end{proof}

The final step of our computations will be to use the intersection table in Proposition~\ref{proposition-intersection-D} to further investigate the divisor classes on $Y$. As a preparation, we consider the image of the curve $N_{\tilde D}$ in $Y$: by construction this image is a curve $N \subset E_P$ which is a bisection of the ruling $E_P \rightarrow R$.  We can make this description more geometric: recall that
$$
E_P \simeq \PP(\Omega_S \otimes \sO_R) \simeq \mathbf P(T_S \otimes \sO_R), 
$$
and by Lemma~\ref{lemma-D-DS} the fibres of the universal family map birationally onto the canonically lifted curves.  Also recall that for a point $x \in R$, there exists a unique element of $|L|$ that is singular in $x$ (the preimage of the tangent line of $B \simeq R$ in $x$); denote this element by $C_x$.  Using the notion of a tangent line of a plane singular curve as explained in \cite[Section~V.4.8]{Per08},
we obtain that
\begin{equation} \label{definition-N}
N = 
\{
v \in \PP(\Omega_S \otimes \sO_R) \simeq P(T_S \otimes \sO_R) \ | \
v \mbox{ is a tangent line of $C_x$, where $x=\pi(v)$} 
\}.
\end{equation}

\begin{lemma} \label{lemmaclassN}
The class of\, $N \subset E_P$ is $2 \tilde R + 72 l_p$, where $l_p$ is a fibre of the ruling $E_P \rightarrow R_P$. 

Moreover we have the following intersection numbers:
$$
E_S \cdot N=0, \ \ \  E_P \cdot N = 36, \ \ \  M \cdot N = 60, \ \ \  D \cdot N = 1728.
$$
\end{lemma}

\begin{proof}
We know that $N \subset E_P$ is a bisection, so we have $N \equiv 2 \tilde R + \lambda l_p$.  We claim that $N$ is disjoint from the negative section $\tilde R \subset E_P$. By Lemma~\ref{lemma-intersectionsY-1}, we have $(\tilde R)_{E_P}^2=-36$, so the claim determines the class of $N \subset E_P$. The intersection numbers then follow from Lemma~\ref{lemma-intersectionsY-1} and the fact that the map $(q \circ \mu_P)|_N\colon N \rightarrow R^\vee$ has degree two and $R^\vee=B^\vee$ has degree 30 (\textit{cf.} Section~\ref{subsection-surface-D}).

{\em Proof of the claim.}  Recall that the tangent map $f^* \Omega_{\PP^2} \rightarrow \Omega_S$ has rank one along $R$, and its image defines a canonical inclusion $\Omega_R \hookrightarrow \Omega_S \otimes \sO_R$ (\textit{cf.} Section~\ref{subsection-elementary-transform}).  The negative section $\tilde R$ is the section corresponding to the induced quotient bundle $\Omega_S \otimes \sO_R \twoheadrightarrow (\Omega_S \otimes \sO_R/\Omega_R)$.  If we write $f$ in local coordinates as $(x,y) \mapsto (x, y^2)$ (so the ramification divisor is $y=0$), the image of $f^* \Omega_{\PP^2} \otimes \sO_R \rightarrow \Omega_S \otimes \sO_R$ is generated by $dx$, so the quotient is given by $dy$. Dually the image of the inclusion $(\Omega_S \otimes \sO_R/\Omega_R)^* \hookrightarrow T_S|_R$ is generated by $\frac{\partial}{\partial y}$. Thus it is sufficient to verify that the tangent lines of the curves $C_x$ are not collinear to $\frac{\partial}{\partial y}$.

If $x \in R \simeq B$ is an inflection point, the curve $C_x$ is cuspidal and has a unique tangent line which coincides with the tangent line of the ramification divisor $R$.
Thus it is generated by $\frac{\partial}{\partial x}$. 

If $x \in R \simeq B$ is not an inflection point, the curve $C_x$ is nodal 
with local equation $y^2-x^2=(y-x)(y+x)$, so none of the tangent lines
is generated by $\frac{\partial}{\partial y}$.
\end{proof}

\begin{remark} \label{remark-ES-irreducible}
We conclude this subsection with an observation indispensable for the proof of Theorem~\ref{theorem-estimate-ray}: we saw in Section~\ref{subsection-geometry-barD} that the intersection of the surfaces $\fibre{p_2}{B^\vee}=\fibre{p_2}{R^\vee}$ and $\fibre{p_1}{B}$ has two components, which in the ruled surface $T$ correspond to the curves $R_T$ and $B_T$. By Corollary~\ref{corollary-BT-irreducible} the curve $B_T$ is irreducible.

Since the double cover $\tilde f$ is bijective along its ramification divisor $\fibre{p}{R}$, this shows that the intersection of $\fibre{q}{R^\vee}$ and $\fibre{p}{R}$ has two irreducible components, the curves $R_P$ and $\fibre{\tilde f}{B_T}$.  In the generic point of $R_P$, each branch of $\fibre{q}{R^\vee}$ meets $\fibre{p}{R}$ transversally, so their strict transforms in $Y$, \textit{i.e.}\ the surfaces $D$ and $E_S$, intersect only along the strict transform of the irreducible curve $\fibre{\tilde f}{B_T}$.  The curve $\fibre{\tilde f}{B_T}$ is not contained in the singular locus of $D$, so the pull-back
$$
\nu^* E_S = \nu^* (\fibre{\tilde f}{B_T})_{\red}
$$
is an irreducible curve in the normalisation $\tilde D$.
\end{remark}

\subsection{Numerical restrictions on pseudoeffective classes}
\label{subsection-numerical}

Our goal in this subsection is to use the description of the surface $\tilde D$ to 
obtain additional information on the effective divisors in $Y$. Somewhat abusively we denote by
$$
\holomd{\nu}{\tilde D}{Y}
$$
the composition of the normalisation $\holom{\nu}{\tilde D}{D}$ with the inclusion $D \hookrightarrow Y$. 

\begin{lemma} \label{lemma-generated}
The image of the pull-back map $\holom{\nu^*}{N^1(Y)}{N^1(\tilde D)}$
is contained in the subspace generated by the classes
$C_{\node,t}, l_{\node,t}, C_{\cusp,t}, l_{\cusp,t}$ and $N_{\tilde D}$ $($\textit{cf.} Proposition~\ref{proposition-intersection-D} for the notation$)$.
\end{lemma}

\begin{proof}
Denote by $\holom{i}{S}{S}$ the involution induced by the double cover $f$. Then $i$ acts via push-forward on the linear system $|L| = (\PP^2)^\vee$.  Since all the elements of $|L|$ are pull-backs from $\PP^2$, this action is trivial.  Thus we have a natural involution $(i, i_*) \colon S \times |L| \rightarrow S \times |L|$ which preserves the universal family of $|L|$, \textit{i.e.}\ the subvariety $\PP(f^* \Omega_{\PP^2}) \subset S \times |L|$.  We denote by $\holom{i_P}{\PP(f^* \Omega_{\PP^2})}{\PP(f^* \Omega_{\PP^2})}$ the induced involution and note that by construction, the fibrations $p$ and $q$ are equivariant with respect to the action of $i_P$. Since $R_P \subset \PP(f^* \Omega_{\PP^2})$ is the locus where $q$ is not smooth, it is is preserved by $i_P$. Thus $i_P$ lifts to an involution $\holom{i_Y}{Y}{Y}$. Since the action of $i_*$ preserves the curve $R^\vee$, the involution $i_Y$ leaves the surface $D$ invariant, so it lifts to an involution $\holom{\tilde i}{\tilde  D}{\tilde D}$. Taking the quotient by these involutions, we obtain a commutative diagram
$$
\xymatrix{
\tilde D \ar[r]^{\nu} \ar[d]^{\tau_D} & Y  \ar[d]^\tau
\\
\tilde D/\langle \tilde i \rangle \ar[r]^{\nu'} \ar[d] & Y/\langle i_Y \rangle \ar[d]
\\
T \ar[r]^{\nu_T} & \PP(\Omega_{\PP^2})\rlap{.}
}
$$
The threefold $Y/\langle i_Y \rangle$ is a (weighted) blow-up of $\PP(\Omega_{\PP^2})$, so it has Picard number three. Since $Y$ also has Picard number three, this shows that the pull-back $\holom{\tau^*}{N^1(Y/\langle i_Y \rangle)}{N^1(Y)}$ is an isomorphism.  Thus it is sufficient to show that the image of the pull-back
$$
\holom{(\nu')^*}{N^1(Y/\langle i_Y \rangle)}{N^1(\tilde D/\langle \tilde i \rangle)}
$$
is generated by the classes $\tau_D(C_{\node,t}), \tau_D(l_{\node,t}), \tau_D(C_{\cusp,t}), \tau_D(l_{\cusp,t})$ and $\tau_D(N_{\tilde D})$. But the surface $\tilde D/\langle \tilde i \rangle$ is a (weighted) blow-up of the ruled surface $T$ with exceptional locus the curves $\tau_D(l_{\node,t})$ and $\tau_D(l_{\cusp,t})$, so its N\'{e}ron--Severi space is generated by these classes.
\end{proof}

\begin{lemma} \label{lemma-pullback}
In the situation of\, Theorem~\ref{theorem-structure-D} and Proposition~\ref{proposition-intersection-D}, we set
$$
\Lnode := \sum_{t \in R, \nu(t) \mbox{ is a node}} l_{\node, t}, 
\qquad
\Lcusp := \sum_{t \in R, \nu(t) \mbox{ is a cusp}} l_{\cusp, t}, 
$$
Denote by $l_D$ a general fibre of the elliptic fibration $\holom{\psi_{\tilde D}}{\tilde D}{R}$.
Then we have
\begin{eqnarray}
\nu^* E_S &=& 2 \ND + 72 l_D - \Lnode + \Lcusp,
\label{pullES}
\\
\nu^* E_P &=& \ND + \Lnode + 2 \Lcusp,
\label{pullEP}
\\
\nu^* M &=& 30 l_D.
\label{pullM}
\end{eqnarray}
\end{lemma}

Before we give the proof of the lemma, let us recall the geometry of the situation: the curves $l_{\node, t}$ and $l_{\cusp, t}$ are the exceptional curves of the blow-up $\holom{\tilde \mu_P}{\tilde D}{\bar D}$ (\textit{cf.} Theorem~\ref{theorem-structure-D}). This blow-up is induced by the blow-up $\holom{\mu_P}{Y}{\PP(f^* \Omega_{\PP^2})}$, and we have
$$
\nu_* l_{\node, t} = l_P = \nu_* l_{\cusp, t}, 
$$
where $l_P$ is an exceptional curve of $\mu_P$. The general fibre of $\holom{\psi_{\tilde D}}{\tilde D}{R}$ maps isomorphically onto the general fibre of $q \circ \mu_P\colon D \rightarrow R^\vee$.  This justifies that we denote them by the same letter $l_D$ (\textit{cf.} Lemma~\ref{lemma-intersectionsY-4}).

\begin{proof}[Proof of Lemma~\ref{lemma-pullback}]
By Lemma~\ref{lemma-generated} we know that the pull-backs can be written as  linear combinations of the classes $C_{\node,t}, l_{\node,t}, C_{\cusp,t}, l_{\cusp,t}$ and $N_{\tilde D}$. These classes are not linearly independent, but by Theorem~\ref{theorem-structure-D} one has
$$
l_D \equiv  C_{\node,t} + 2 l_{\node,t} \equiv  C_{\cusp,t} + l_{\cusp,t}.
$$
Thus it is straightforward to see that the classes $\ND, \ l_D, \ l_{\node,t}, \ l_{\cusp,t}$ form a basis of the vector space. In particular we can verify the formulas by computing the intersection numbers with the elements of the basis.

By the discussion before this proof, we see that the intersection numbers with the left-hand sides are determined by Lemmas~\ref{lemma-intersectionsY-4} and~\ref{lemmaclassN} and the projection formula.  The intersections with the right-hand sides are determined by Proposition~\ref{proposition-intersection-D}.
\end{proof}

\begin{remark} \label{remark-mori-cone-D}
The Mori cone of the surface $\tilde D$ is quite complex: the irreducible curves $\ND, C_{\node,t}, l_{\node,t},$ $C_{\cusp,t}, l_{\cusp,t}$ have negative self-intersection, so they generate extremal rays in the Mori cone; \textit{cf.} \cite[Lemma 1.22]{KM98}.  By Remark~\ref{remark-ES-irreducible} the effective divisor $\nu^* E_S$ is irreducible, and by Equation~\eqref{pullES} the class of this divisor is not in the cone generated by the classes $\ND, C_{\node,t}, l_{\node,t}, C_{\cusp,t}, l_{\cusp,t}$. Since $(\nu^* E_S)^2 = E_S^2 \cdot D=-288<0$ by Lemma~\ref{lemma-intersectionsY-3}, the irreducible curve $\nu^* E_S$ thus generates another extremal ray in Mori cone.
\end{remark}

With the description of the geometry of $E_S$ and $D$ in mind, we obtain strong obstructions to the effectivity of a class in $\PP(\Omega_S)$. 

\begin{proof}[Proof of Theorem~\ref{theorem-estimate-ray}]
	Let $Z_S \subset \PP(\Omega_S)$ be a prime divisor, and let $Z\subset Y$ be its strict transform. The statement is clear for $Z=D$ and $Z=E_P$, so we assume that $Z$ is distinct from these two surfaces.  Note that we can assume that the class of $Z$ is not a linear combination $Z = x D+ y E_P$: computing the intersections with $l_D$ and $l_P$ (see Lemma~\ref{lemma-intersectionsY-4}), we obtain that $x \geq 0$ and $y \geq 0$. Thus we have $\lambda \geq 1.8$. Thus, up to a multiple, we can denote by $\eta:=E_S+ x D+ y E_P$ the cohomology class of $Z$ in $Y$.  Since $E_S$ is $\mu_S$-exceptional, the prime divisor $Z$ is not equal to $E_S$, and its restriction to $E_S$ is effective. By Corollary~\ref{corollarypseffES}, this tells that $Z|_{E_S}$ is also nef, and in particular $\eta^2 \cdot E_S\geq 0$. Using Lemmas~\ref{lemma-intersectionsY-2} and~\ref{lemma-intersectionsY-3} one can easily see that this implies that
	\begin{equation} \label{help1}
	\frac{1-4x}{3} \leq y.
	\end{equation}
	
We claim that $\alpha=\nu^* E_S +0.5 \Cnode$ is nef. 

{\em Proof of the claim.} 
Using Equation \eqref{pullES} and Proposition~\ref{proposition-intersection-D}, one sees that $\alpha\cdot \Cnodei=0$ and $\alpha\cdot \nu^* E_S=360$. Since by Remark~\ref{remark-ES-irreducible} the curve $\nu^* E_S $ is irreducible, the class $\alpha $ is non-negative along the irreducible components of its support. This implies the claim. 
	
Since $\alpha$ is nef, using again Equation \eqref{pullES} and Proposition~\ref{proposition-intersection-D}, we obtain
	$$
	0\leq \alpha \cdot \nu^* \eta 	=72(5-8x+29y).
	$$
Combining this condition with the inequality \eqref{help1} 	
gives $\frac{y}{x} \geq -\frac{3}{11}$. 
	
	Since $(\mu_S)_* E_S = 0$, $(\mu_S)_* E_P \equiv \pi^* 3 L$
	and $(\mu_S)_* M \equiv \zeta_S + \pi^* 2 L$ (\textit{cf.} Remark~\ref{remark-zeta-twoL}), 
	a straightforward com\-pu\-ta\-tion (\textit{cf.} Lemma~\ref{lemmaclasstransform}) gives
$$
(\mu_S)_* \eta  \equiv 30x\left(\zeta_S+\left(1.8+\frac{y}{10x}\right)\pi^*L\right).
$$
	
Since the ratio $\frac{y}{10x}$ is bounded from below by $-\frac{3}{110}$, the cohomology class of the prime divisor $Z_S$ in $\PP(\Omega_S)$ can thus be written as $[Z_S] \equiv a(\zeta_S+\lambda \pi^*L)$ with $\lambda \geq\frac{9}{5}-\frac{3}{110}=\frac{39}{22}$, as desired.
\end{proof}

\begin{remark*}
The argument above and thus the bound in Theorem~\ref{theorem-estimate-ray} depend on our choice of the class $\alpha$, so it likely that a full description of the nef cone of $\tilde D$ would lead to a better estimate. However, it is not clear that even such an optimised necessary condition determines the pseudoeffective threshold. 
\end{remark*}

\section{Positivity results} 
\label{section-positivity} 

In this whole section we work in the same set-up as in Section~\ref{subsection-elementary-transform} (\textit{cf.}~the summary in Diagram \eqref{diagram}). We denote by 
$$
\holomd{\varphi}{Y}{S}
$$
the composition $\pi \circ \mu_S=p \circ \mu_P$.  The pseudoeffective
cone of the threefold $\PP(f^* \Omega_{\PP^2})$ is generated by the
classes $p^* L$ and $q^* H_2$, where $H_2$ is the hyperplane class on
$(\PP^2)^\vee$. The picture becomes significantly less trivial when we
consider the pseudoeffective cone of $\PP(\Omega_S)$ or $Y$. Given a
divisor class $Z$ on $Y$, we write
\begin{equation} \label{notation-Z}
Z \equiv a M + b \varphi^* L - m E_P.
\end{equation}
If $Z \subset Y$ is an effective $\R$-divisor such that the support does not contain 
the exceptional divisor $E_P$, its image 
$Z_P \subset \PP(f^* \Omega_{\PP^2})$ has class $Z_P \equiv a q^* H_2 + b p^* L$
with $a \geq 0, b \geq 0$. Moreover, the effective divisor $Z_P$ has multiplicity $m$ along the curve $R_P$.

\subsection{Necessary conditions}

In this subsection we collect some basic necessary conditions. As an application we prove the second part of Theorem~\ref{theorem-main2}.

\begin{lemma} \label{lemmaclasstransform}
Let $Z \subset Y$ be an effective $\R$-divisor 
such that the support does not contain 
the exceptional divisors $E_P$ and $E_S$. 
Let $Z_S \subset \PP(\Omega_S)$ be the image of\, $Z$.
Then, using the notation \eqref{notation-Z}, one has
$$
Z_S \equiv a \zeta_S + (2a+b-3m)  \pi^* L.
$$
\end{lemma}

\begin{proof}
Recall that $E_P+E_S \equiv 3 \varphi^* L$, so
$$
Z \equiv a M + \tfrac{1}{3}b E_S + \left(\tfrac{1}{3}b - m\right) E_P.
$$
Since $(\mu_S)_* E_S \equiv 0, \ (\mu_S)_* E_P \equiv \pi^* 3 L$
and $(\mu_S)_* M = \zeta_S + \pi^* 2 L$, we obtain
\[\pushQED{\qed} 
[Z_S] \equiv a \zeta_S + (2a+b-3m) \pi^* L.\qedhere
\popQED
\]\renewcommand{\qed}{}
\end{proof}

\begin{remark*}
The divisor $D \subset Y$ introduced in Proposition~\ref{proposition-class-D} is the strict transform of $\fibre{q}{R^\vee}$,
for which we have $a=30, b=0, m=2$. Thus Lemma~\ref{lemmaclasstransform} is a technical extension of the formula in Theorem~\ref{theorem-main1}. 
\end{remark*}

\begin{corollary}
Let $Z_S \subset \PP(\Omega_S)$ be an effective $\R$-divisor
such that
the class of $Z_S$ is not contained in the cone generated by $D_S$ and $\pi^* L$.
Then, using  the notation  \eqref{notation-Z}, we have
\begin{equation} \label{conditionextremal}
\frac{2a+b-3m}{a} < 1.8 
\quad \Longleftrightarrow \quad \tfrac{1}{5}a+b < 3m. 
\end{equation}
\end{corollary}

\begin{proof} Recall that by Theorem~\ref{theorem-main1}, we have
$[D_S] \equiv 30 (\zeta_S + 1.8 \pi^* L)$. Thus we simply apply
  Lemma~\ref{lemmaclasstransform} to the strict transform of $Z_S$.
\end{proof}

\begin{lemma} \label{lemma-necessary-conditions}
Let $Z \subset Y$ be an effective $\R$-divisor 
such that the support does not contain 
the exceptional divisors $E_P$ and $E_S$.
Then, using the notation \eqref{notation-Z}, we have
\begin{equation} \label{inequality-over-R}
a \geq m \qquad \mbox{and} \qquad \tfrac{1}{2}a+b \geq \tfrac{9}{2} m.
\end{equation}
Assume furthermore that there exists a canonical lifting $\CnodeS$
of a rational nodal curve $C \in |L|$ $($\textit{cf.} Notation~\ref{notation-node-cusp}\,$)$ that is not contained in $Z_S$. Then one has 
\begin{equation} \label{inequality-bitangent}
b \geq 2m.
\end{equation}
\end{lemma}

As preparation let us note that the class of the curve $R_P \subset \fibre{p}{R}$ is given by
$$
R_P = (q^* H_2 + 4 p^* L)|_{\fibre{p}{R}}.
$$
Indeed since $p(R_P)=R$ (resp.\ $q(R_P)=R^\vee$), we know that $p^* L \cdot R_P=6$ and $q^* H_2 \cdot R_P=30$. By Lemma~\ref{lemma-intersectionsY-2} we have $(q^* H_2)^2 \cdot \fibre{p}{R} = M^2 \cdot E_S= 6$, which implies the result.

\begin{proof}[Proof of Lemma~\ref{lemma-necessary-conditions}]
The support of $Z$ does not contain $E_S$, so the support of its image $Z_P$ 
does not contain $\mu_P(E_S)=\fibre{p}{R}$. Since $Z_P$ has multiplicity $m$ along $R_P$, the intersection $Z_P \cap \fibre{p}{R}$ contains the curve $R_P$ with multiplicity at least $m$.
Thus if $f$ is a general fibre of the ruling $\fibre{p}{R} \rightarrow R$, then 
$(Z_P \cap \fibre{p}{R}) \cdot f \geq m$. Since $Z_P \equiv a q^* H_2 + b p^* L$ we obtain that $a \geq m$.

By what precedes the class
$$
[Z_P \cap \fibre{p}{R}] - m R_P \equiv  [(a-m) q^* H + (b-4m) p^* L]|_{\fibre{p}{R}}
$$
is effective. The isomorphism $E_S \rightarrow \fibre{p}{R}$ maps the
curve $\tilde R$ onto $R_P$, so the class $\tilde R - 27 l_S$ maps
onto the class $(q^* H_2 - \frac{1}{2} p^* L)|_{\fibre{p}{R}}$. By
Corollary~\ref{corollarypseffES} the class $\tilde R - 27 l_S$ is nef,
so
$$
(q^* H_2 - \tfrac{1}{2} p^* L) \cdot [(a-m) q^* H_2 + (b-4m) p^* L] \cdot \fibre{p}{R} \geq 0.
$$
Using the intersection numbers from Lemma~\ref{lemma-intersectionsY-2}, we obtain the second inequality in \eqref{inequality-over-R}.

For the proof of the last statement, we recall that by Lemma~\ref{lemma-D-DS}, the curve $\CnodeS$ is the strict transform of the corresponding fibre $\Cnode$ of the universal family $\PP(f^* \Omega_{\PP^2}) \rightarrow |L| \simeq (\PP^2)^\vee$. Since the generic point of $\CnodeS$ is contained in the locus where $\mu_P \circ \mu_S^{-1}$ is an isomorphism, its strict transform $\Cnode$ is not contained in $Z_P$.  The curve $\Cnode$ has two singular points, which are both on the curve $R_P$ (\textit{cf.} the proof of Lemma~\ref{lemma-singularities-D}).  Since $Z_P$ has multiplicity $m$ along $R_P$, we obtain that $Z_P \cdot \Cnode \geq 4m$. Since $Z_P \equiv a q^* H_2 + b p^* L$ and $\Cnode$ is contracted by $q$, this yields the inequality \eqref{inequality-bitangent}.
\end{proof}

\begin{corollary} \label{corollary-extend-to-pseff}
Let $Z$ be a pseudoeffective $\R$-divisor class on $Y$ that is modified nef, \textit{cf.} \cite[Definition~2.2]{Bou04}; write its divisor class as $Z \equiv a M + b \varphi^* L - m E_P$. Then we have
$$
a \geq m \qquad \mbox{and} \qquad \tfrac{1}{2}a+b \geq \tfrac{9}{2} m.
$$
\end{corollary}

\begin{remark*}
The property that a divisor class is modified nef is equivalent to its being in the closure of the cone of mobile divisor classes; \textit{cf.} \cite[Section~5.1]{Bou04}.
\end{remark*}

\begin{proof}
Since $Z$ is modified nef, it is in the closure of the mobile cone. Thus we can find a sequence of effective mobile $\Q$-divisors $Z_t$ such that their classes converge to $Z$.  Since the decomposition of the divisor class in the basis $M, \varphi^* L , E_P$ is unique, this implies that the coefficients $a_t, b_t, m_t$ converge to $a, b, m$.  The divisor $Z_t$ being mobile, its support does not contain the exceptional divisors $E_P$ and $E_S$. Thus Lemma~\ref{lemma-necessary-conditions} yields $a_t \geq m_t$ and $\frac{1}{2}a_t+b_t \geq \frac{9}{2} m_t$. We conclude by passing to the limit.
\end{proof}

\begin{corollary} \label{corollarynodal}
Let $Z_S \subset \PP(\Omega_S)$ be an effective $\R$-divisor such that
$Z_S = a (\zeta + \lambda \pi^* L)$ with $\lambda < 1.8$. 
Let $C_S \subset \PP(\Omega_S)$ be a canonical lifting
of a rational nodal curve $C \in |L|$. Then $C_S \subset Z_S$.
\end{corollary}

\begin{proof}
We argue by contradiction and assume that $C_S \not\subset Z_S$.  By \eqref{inequality-bitangent} this implies $b \geq 2m$. Moreover by \eqref{inequality-over-R} we always have $a \geq 9m-2b$. Thus we have
$$
\tfrac{1}{5} a + b \geq \tfrac{9}{5} m + \tfrac{3}{5} b \geq 3 m,
$$
which contradicts \eqref{conditionextremal}.
\end{proof}

\subsection{Sufficient conditions}

In this subsection we will use Boucksom's divisorial Zariski decomposition \cite[Theorem~4.8]{Bou04}: given a pseudoeffective $\R$-divisor class $Z$ on $Y$, there exists a decomposition $Z_++Z_-$ such that $Z_+$ is modified nef, \textit{cf.} \cite[Definition~2.2]{Bou04}, and $Z_-$ is an effective $\R$-divisor.

\begin{lemma} \label{lemmacomputationovercurve}
Let $Z$ be a pseudoeffective $\R$-divisor class on $Y$, and let 
$Z_++Z_-$ be its divisorial Zariski decomposition.
Assume that the support of\, $Z_-$ does not contain 
the exceptional divisors $E_P$ and $E_S$.
Then one has
$$
\varphi^* L \cdot Z^2 > 0
$$
unless $Z \equiv b \varphi^* L$. 
\end{lemma}

\begin{remark} \label{remark-restriction-irreducible}
We will need the following elementary remark: let $Z' \subset Y$ be a prime divisor that surjects onto the surface $S$. For $d \gg 0$ let $A \in |dL|$ be a general element.  Then the intersection $Z' \cap \fibre{\varphi}{A}$ is an irreducible curve. Namely the linear system $\varphi|_{Z'}^* |dL|$ is globally generated, so a general element does not have an irreducible component that is in the non-normal locus of $Z'$. Thus we can assume without loss of generality that $Z'$ is normal. Yet then a general element is normal by \cite[Theorem~1.7.1]{BS95}, hence smooth. Since it is connected by the Kawamata--Viehweg vanishing theorem, a general element is irreducible.
\end{remark}

\begin{proof}
We first prove the statement in the case where $Z$ is a prime divisor. By our assumption on $Z_-$, the divisor $Z$ is then distinct from $E_S$ and $E_P$.  We use the notation \eqref{notation-Z}; \textit{i.e.}\ we write
$$
Z \equiv a M + b \varphi^* L - m E_P.
$$
Note that if $a=0$, then $Z \cdot l_s = - m$. Since $Z$ is distinct from $E_S$, we obtain that $m=0$, so $Z \equiv b \varphi^* L$. Thus we can assume that $a>0$.

Using the formulas in Lemma~\ref{lemma-intersectionsY-3}, we obtain
$$
\varphi^* L \cdot N^2 = 2 (a^2+2 ab-3m^2).
$$
Since $a>0$ we have $a^2+ab > a \cdot (\frac{1}{2}a+b)$.
Applying  \eqref{inequality-over-R} twice, we obtain
$$
a \cdot (\tfrac{1}{2}a+b) \geq \tfrac{9}{2} m^2.
$$
Thus we have 
$$
a^2+ab-3m^2 > 1.5 m^2 \geq 0.
$$

The same argument holds if $Z$ is a pseudoeffective divisor class that is modified nef. Indeed the argument above only uses \eqref{inequality-over-R}.
By Corollary~\ref{corollary-extend-to-pseff}, this inequality also holds for pseudoeffective divisors classes that are modified nef.

Now we consider the general case: we write $Z=Z_+ + \sum \lambda_i Z_i$,
where $\lambda_i>0$ and the $Z_i$ are the irreducible components of the negative part. 
Fix $d \gg 0$ such that for a very general element $A \in |dL|$, the intersection
$C_i:=Z_i \cap \fibre{\varphi}{A}$ is irreducible (this is possible by Remark~\ref{remark-restriction-irreducible}) and $Z_+|_{\fibre{\varphi}{A}}$ is nef (this is possible since the non-nef locus of $Z_+$ consists of at most countably many curves; \textit{cf.} \cite[Definition~3.3]{Bou04}).

Since the divisors $Z_i$ are distinct, the irreducible
curves $C_i =Z_i \cap \fibre{\varphi}{A}$ are distinct. Thus we have
$$
\left(Z|_{\fibre{\varphi}{A}}\right)^2
=
\left(Z_+|_{\fibre{\varphi}{A}} + \sum \lambda_i C_i\right)^2
\geq
\left(Z_+|_{\fibre{\varphi}{A}}\right)^2 + \sum \lambda_i^2 C_i^2.
$$
By what precedes all the terms are non-negative, and the sum is positive unless all the classes are  pull-backs from $S$. 
\end{proof}

We can now give a bigness criterion on $Y$. For the proof of Theorem~\ref{theorem-main2}, we will need to work with $\R$-divisor classes. This leads to some additional technical effort.

\begin{lemma} \label{lemmavanishcohomology-R}
Let $Z \subset Y$ be a pseudoeffective $\R$-divisor class, and let 
$Z_++Z_-$ be its divisorial Zariski decomposition. Assume that 
the support of\, $Z_-$ does not contain 
the exceptional divisors $E_P$ and $E_S$.
Assume that there exists a rational number $\varepsilon>0$ such that
$(Z + \varepsilon E_S) \cdot l_S > 0$ and $(Z+\varepsilon E_S)^3>0$.
Then $Z+\varepsilon E_S$ is big.
\end{lemma}

\begin{proof}

{\em Step 1. Positivity properties of $Z+\varepsilon E_S$.}  The relative Picard number of the conic bundle $\holom{\varphi}{Y}{S}$ is two, and the relative Mori cone is generated by the curve classes $l_P$ and $l_S$.  By assumption,we have $(Z+\varepsilon E_S) \cdot l_S>0$. Since $E_P$ is not contained in the support of $Z_-$ and the deformations of $l_P$ cover a divisor, we have
$$
Z \cdot l_P \geq Z_- \cdot l_P \geq 0.
$$
Thus we also have $(Z+\varepsilon E_S) \cdot l_P > 0$, and the divisor class
$Z+\varepsilon E_S$ is $\varphi$-ample.

Let $A \in |dL|$ be a very general element for some $d \gg 0$.
Set $Y_A := \fibre{\varphi}{A}$, and denote by $\holom{\varphi_A}{Y_A}{A}$
the restriction of $\varphi$. We denote the restriction of a divisor class $T$ 
to $Y_A$ by $T_A$. We claim that $(Z+\varepsilon E_S)_A$ is an ample $\R$-divisor class.

{\em Proof of the claim.}  Note that we do not have $Z \equiv b \varphi^* L$ since otherwise $(Z + \varepsilon E_S) \cdot l_S = -\varepsilon< 0$. Thus we can apply Lemma~\ref{lemmacomputationovercurve} and obtain $\varphi^* L \cdot Z^2>0$.  Since the class of $A$ is a positive multiple of $L$, this implies $Z_A^2>0$.

By construction $Y_A$ is the blow-up of $\PP(\Omega_S \otimes \sO_A)$ in the finite set $R_S \cap \fibre{\pi}{A}$. Since $E_S$ is the exceptional divisor of $\mu_S$, the exceptional locus of
$$
\mu_{S,A}\colon Y_A \longrightarrow \PP(\Omega_S \otimes \sO_A)
$$
is equal to the support of $E_S \cap Y_A$, which is a disjoint union of curves numerically equivalent to $l_S$. Since by assumption $(Z+\varepsilon E_S) \cdot l_S>0$, we obtain that
$$
(Z+\varepsilon E_S)_A \cdot (E_S)_A > 0.
$$
Thus we have
$$
(Z+\varepsilon E_S)_A^2 > Z_A \cdot  (Z+\varepsilon E_S)_A \geq Z_A^2, 
$$
where in the last inequality, we used again that $E_S$ is not in the negative part of $Z$.  Since $Z_A^2>0$ this finally shows that
\begin{equation} \label{restriction-big-R}
(Z+\varepsilon E_S)_A^2 > 0.
\end{equation}
Now let $Z=Z_+ + \sum \lambda_i Z_i$ be the divisorial Zariski decomposition. By assumption none the prime divisors $Z_i$ coincides 
with $E_S$ or $E_P$. Thus the intersection $Z_i \cap \fibre{\varphi}{A}$
is either a union of general $\varphi_A$-fibres or an irreducible curve $C_i$ (\textit{cf.} Remark~\ref{remark-restriction-irreducible}). Moreover by Lemma~\ref{lemmacomputationovercurve} we have
$$
C_i^2 = d \varphi^*L Z_i^2 \geq 0, 
$$
and equality holds if and only if $Z_i \equiv b_i \varphi^* L$. In both cases, the curve $C_i$ is a nef divisor on $Y_A$.  Since $Z_+$ is modified nef and $A$ is very general, the restriction $(Z_+)_A$ is nef .  In conclusion we obtain that
$$
Z_A \equiv (Z_+)_A + \sum \lambda_i C_i
$$
is a nef divisor. 

In view of \eqref{restriction-big-R} and the Nakai--Moishezon criterion, \textit{cf.} \cite[Theorem~2.3.18]{Laz04a}, the claim follows if we show that $(Z+\varepsilon E_S)_A \cdot G>0$ for every irreducible curve $G \subset Y_A$.  Arguing by contradiction, we assume that there exists an irreducible curve $G$ such that $(Z+\varepsilon E_S)_A \cdot G \leq 0$.

Since $(Z + \varepsilon E_S) \cdot l_S > 0$ we have $G \not\subset E_S$. Since $Z_A$ is nef and $(Z+\varepsilon E_S)_A \cdot G \leq 0$, we obtain that $E_S \cdot G=0$.  Hence $G$ is disjoint from the exceptional locus of $\mu_{S,A}$. Thus we have
$$
0 \geq (Z+\varepsilon E_S)_A \cdot G = \left(\mu_{S,A}\right)_* (Z+\varepsilon E_S)_A \cdot (\mu_{S,A})_* G.
= \left(\mu_{S,A}\right)_* Z_A \cdot (\mu_{S,A})_* G.
$$ 
We claim that $(\mu_{S,A})_* Z_A$ is an ample class; this yields the desired contradiction.

For the proof of the claim, note that the push-forward $(\mu_S)_* Z$ is pseudoeffective and not a pull-back from~$S$, so it is collinear to $\zeta_S+\lambda \pi^* L$ for some $\lambda \in \Q$.  Since $\zeta_S$ is not pseudoeffective, we have $\lambda>0$.
Yet the cotangent bundle $\Omega_S$ is stable, so if $A$ is a sufficiently positive hyperplane section, the restricted vector bundle $\Omega_S \otimes \sO_A$ is stable with trivial determinant. Thus the restriction of the tautological class $\zeta_S$ to $\fibre{\pi}{A}$ is nef. Since $\lambda>0$ this shows that the restriction of $\zeta_S+\lambda \pi^* L$ to $\fibre{\pi}{A}$ is ample.  This proves the claim.

{\em Step 2. Approximation.}
Since ampleness is an open property,  by Step~1, we can  find  real numbers $0 \leq \delta_i \ll 1$ such that
$$
Z^\delta := Z-\delta_1 \varphi^* L - \delta_2 E_P - \delta_3 E_S
$$
is a $\Q$-divisor class with the following properties:
\begin{itemize}
\item $Z^\delta+\varepsilon E_S$ is $\varphi$-ample;
\item $(Z^\delta+\varepsilon E_S)_A$ is ample;
\item $(Z^\delta+\varepsilon E_S)^3>0$.
\end{itemize}
We claim that such a $\Q$-divisor class is big. Since $\delta_1 \varphi^* L + \delta_2 E_P + \delta_3 E_S$ is effective, this will finish the proof of the lemma.

{\em Proof of the claim.}
Our goal is to show that for $m$ sufficiently high and divisible, we have
$$
H^j(Y, \sO_Y(m(Z^\delta+\varepsilon E_S))) = 0 
$$
for $j=2,3$. Then we have
$$
h^0(Y, \sO_Y(m(Z^\delta+\varepsilon E_S))) 
\geq
\chi(Y, \sO_Y(m(Z^\delta+\varepsilon E_S))),
$$
so the bigness of $Z^\delta+\varepsilon E_S$ follows from the asymptotic Riemann--Roch theorem and the property \mbox{$(Z^\delta+\varepsilon E_S)^3>0$.}

Since $Z^\delta+\varepsilon E_S$ is $\varphi$-ample, we know by relative Serre vanishing that $R^1 \varphi_* \sO_Y(m(Z^\delta+\varepsilon E_S))=0$ for sufficiently high and divisible $m$. By cohomology and base change, \textit{cf.} \cite[Theorem~II.12.11]{Har77}, the direct image sheaf $\varphi_* \sO_Y(m(Z^\delta+\varepsilon E_S))$ is locally free and commutes with base change.

Since the higher direct images vanish, we have
$$
H^j(Y, \sO_Y(m(Z^\delta+\varepsilon E_S))) \simeq H^j(S, \varphi_* \sO_Y(m(Z^\delta+\varepsilon E_S)));
$$
in particular $H^3(Y, \sO_Y(m(Z^\delta+\varepsilon E_S)))=0$. 
Moreover, since  $\varphi_* \sO_Y(m(Z^\delta+\varepsilon E_S))$ is locally free, Serre duality applies:
$$
H^2(S, \varphi_* \sO_Y(m(Z^\delta+\varepsilon E_S))) = H^0(S, (\varphi_* \sO_Y(m(Z^\delta+\varepsilon E_S)))^{*}). 
$$
Hence the vanishing follows if we show that
$$
H^0(A, (\varphi_* \sO_Y(m(Z^\delta+\varepsilon E_S)))^{*} \otimes \sO_A) = 0.
$$
Since the direct image sheaf has the base-change property, we have
$$
\varphi_* \sO_Y(m(Z^\delta+\varepsilon E_S)) \otimes \sO_A
\simeq
(\varphi_A)_* \sO_{Y_A}(m(Z^\delta+\varepsilon E_S)).
$$
Yet $(Z^\delta+\varepsilon E_S)_A$ is ample, so  the direct image sheaf $(\varphi_A)_* \sO_{Y_A}(m(Z^\delta+\varepsilon E_S))$ is ample for $m \gg 0$ by \cite[Corollary~2.11, Theorem~3.1.]{Anc82}.  By \cite[Example~6.1.4]{Laz04b}, 
this implies that its dual $(\varphi_* \sO_Y(m(Z^\delta+\varepsilon E_S)))^{*} \otimes \sO_A$ does not have global sections.
\end{proof}

The next statement is an immediate application of Lemma~\ref{lemmavanishcohomology-R} and a first step towards Theorem~\ref{theorem-main2}. 

\begin{corollary} \label{corollary-surprise}
The divisor class $D+4E_S$ is big. In particular, its push-forward
$(\mu_S)_* (D+4E_S) \equiv 30 (\zeta_S+1.8 \pi^* L)$ is a big divisor class.
\end{corollary}

\begin{proof}
We are done if we verify the conditions of Lemma~\ref{lemmavanishcohomology-R}: using the formulas in Lemma~\ref{lemma-intersectionsY-3}, we obtain that $(D+4 E_S)^3 = 10242>0$.  The divisor $D$ is effective and prime, so the exceptional divisors $E_P$ and $E_S$ are not in the support of the negative part of its divisorial Zariski decomposition.  Finally by Lemma~\ref{lemma-intersectionsY-4} one has $(D+4E_S) \cdot l_S=24>0$.\end{proof}

\begin{proof}[Proof of Theorem~\ref{theorem-main2}]
For simplicity of  notation, we will work with the rounded value $1.7952024$.

We argue by contradiction and assume that $\lambda \geq 1.7952024$ for all prime divisors in $\PP(\Omega_S)$. Let $Z$ be a pseudoeffective divisor class on $Y$ such that $(\mu_S)_* Z:=Z_S$ generates the second extremal ray of the pseudoeffective cone of $\PP(\Omega_S)$. Since $(\mu_S)_* Z$ is not big, the divisor class $Z+\varepsilon E_S$ is not big for every $\varepsilon \geq 0$.

First assume that the negative part the divisorial Zariski decomposition of $Z$ contains the exceptional divisor $E_P$ or $E_S$; then we can write $Z=Z'+\alpha E_P+\beta E_S$ with $Z'$ pseudoeffective and $\alpha \geq 0$ and $\beta \geq 0$.  Since
$$
(\mu_S)_* Z \equiv (\mu_S)_* Z' + \alpha (\mu_S)_* E_P
\equiv  (\mu_S)_* Z' + 3 \alpha \pi^* L 
$$
generates an extremal ray, we see that $\alpha=0$. Since $(\mu_S)_* Z = (\mu_S)_* Z'$, we can replace $Z$ with $Z'$ and assume without loss of generality that 
the negative part of $Z$ does not contain $E_P$ or $E_S$.

Since $\mu_S$ contracts the extremal ray generated by $l_S$, we have $Z \cdot l_S=0$ if and only if $Z=\mu_S^* Z_S$.  Since $Z_S = (\mu_S)_* Z$ generates the second extremal ray of the pseudoeffective cone, we have $Z_S \equiv a (\zeta_S+\lambda \pi^* L)$ with $\lambda<3$. In particular, $Z_S \cdot R_S<0$, so $Z|_{E_S}=(\mu_S|_{E_S})^* Z_S|_{R_S}$ is not pseudoeffective, which gives a contradiction.

Thus we have $Z \cdot l_S>0$. We claim that there exists an $\varepsilon>0$ such that $(Z+\varepsilon l_S)>0$ and $(Z+\varepsilon E_S)^3>0$. Then we know by Lemma~\ref{lemmavanishcohomology-R} that $Z+\varepsilon E_S$ is big, which gives a contradiction.

{\em Proof of the claim.}
A positive multiple of $Z+\varepsilon E_S$ can be written as
$$
E_S + x D + \eta x E_P.
$$
Since $(\mu_S)_* (E_S + x D + \eta x E_P) \equiv 30 x (\zeta_S + (1.8+ \frac{\eta}{10 x}) \pi^*L)$ our assumption and Corollary~\ref{corollary-surprise} imply that
$$
0 > \eta \geq -0.047976. 
$$
By Lemma~\ref{lemma-intersectionsY-4} the condition $(Z+\varepsilon E_S) \cdot l_S>0$ is equivalent to $x > \frac{1}{28+\eta}$. Since $\eta \geq -0.047976$ this holds if $x \geq 0.03577$.  Applying  Lemma~\ref{lemma-intersectionsY-3} again we obtain that $(E_S+xD+\eta x E_P)^3$ is equal to
\begin{equation} \label{cubic-polynomial}
x^3\left(-10224+9936\eta-2268\eta^2-72\eta^3\right)+x^2\left(6048+4752\eta+162\eta^2\right)+x\left(-864-108\eta\right)+18. 
\end{equation}
Set $y=\eta x$; Figure~\ref{figure-cubic} shows the planar cubic defined by this polynomial.  Hence for $\eta \geq -0.047976$, we can find an $x \geq 0$ that satisfies both conditions.\footnote{More formally, we could use the classical formulas for the roots of cubic polynomials to show that for fixed $\eta \geq -0.047976$, the cubic polynomial has a root satisfying the inequality $x \geq 0.03577$. We refrain from giving the tedious details of such a proof.}

\begin{figure}[H] 
\begin{center}
\includegraphics[height=7cm,width=6cm]{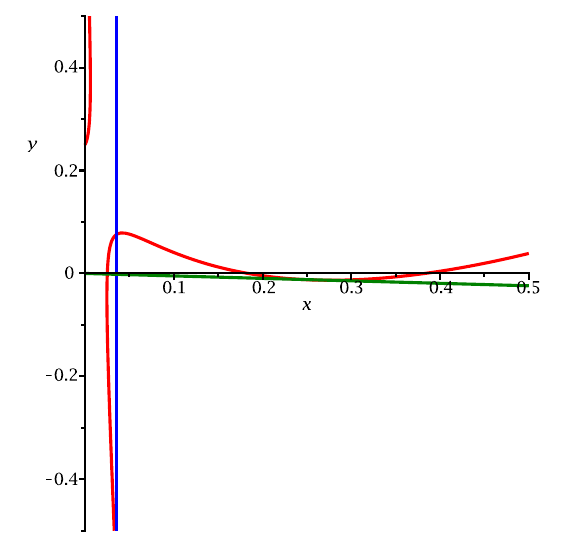}
\end{center}
\vspace{-0.5cm}
\caption{Vanishing set of \eqref{cubic-polynomial} in red (note that since $E_S^3>0$, the origin lies in a connected component of the plane where \eqref{cubic-polynomial} is positive),  line $x = 0.03577$ in blue and line $y=-0.047976 x$ in green.}\label{figure-cubic}
\end{figure}
\vspace{-1.25cm}
\end{proof}

\begin{remark}
Theorem~\ref{theorem-main2} may be a rather small improvement of Corollary~\ref{corollary-surprise}, yet in view of Theorem~\ref{theorem-estimate-ray} it is also clear that there is not much space left. In fact more should be true: by Corollary \ref{corollary-surprise} the divisor $D+4E_S$ (which in Figure~\ref{figure-cubic} corresponds to the point $x=0.25, y=0$) is big. If it is also nef, the estimate in Theorem~\ref{theorem-estimate-ray} can be improved to $\lambda \geq \frac{25}{14} \approx 1.7857$.
\end{remark}

%%%%%%%%%%%%%%%%%%%%%
% References
%%%%%%%%%%%%%%%%%%%%%

\newcommand{\etalchar}[1]{$^{#1}$}
\providecommand{\bysame}{\leavevmode\hbox to3em{\hrulefill}\thinspace}
\providecommand{\MR}{\relax\ifhmode\unskip\space\fi MR }
\providecommand{\MRhref}[2]{%
  \href{http://www.ams.org/mathscinet-getitem?mr=#1}{#2}
}
\providecommand{\href}[2]{#2}

\end{document}